\documentclass[final,onefignum,onetabnum]{siamart190516}

\usepackage{amsmath,amssymb}
\usepackage{mathtools}
\usepackage{lipsum}
\usepackage{amsfonts}
\usepackage{graphicx}
\usepackage{epstopdf}
\usepackage{algorithmic}
\ifpdf
  \DeclareGraphicsExtensions{.eps,.pdf,.png,.jpg}
\else
  \DeclareGraphicsExtensions{.eps}
\fi

\usepackage{enumitem}
\setlist[enumerate]{leftmargin=.5in}
\setlist[itemize]{leftmargin=.5in}

\newsiamremark{remark}{Remark}
\newsiamremark{hypothesis}{Hypothesis}
\crefname{hypothesis}{Hypothesis}{Hypotheses}
\newsiamthm{claim}{Claim}
\newsiamthm{corollar}{Corollary}

\headers{}{}

\title{Unstabilized Hybrid High-Order method for a class\\ of degenerate convex minimization problems\thanks{Submitted to the editors \today.
\funding{This work has been supported by the Deutsche Forschungsgemeinschaft (DFG) in the Priority Program 1748 Reliable simulation techniques in solid mechanics:
Development of non-standard discretization methods, mechanical and mathematical analysis under the project CA 151/22. The second author is also supported by the Berlin Mathematical School. 
The authors thank the anonymous referees for suggestions that led to \Cref{rem:MFEM}--\ref{rem:hybridization}, \ref{rem:pLaplace}, and \ref{rem:LEB}.}}}

\author{C.~Carstensen\thanks{Department of Mathematics, Humboldt-Universit\"at zu Berlin, Germany 
  (\email{cc@math-hu.berlin.de}, \email{tranngoc@math-hu.berlin.de}).}
\and N.~T.~Tran\footnotemark[2]}

\usepackage{amsopn}

\DeclareMathOperator{\D}{D}
\newcommand{\R}{\ensuremath{\mathbb{R}}}
\newcommand{\M}{\ensuremath{\mathbb{M}}}
\newcommand{\Tcal}{\ensuremath{\mathcal{T}}}
\newcommand{\Fcal}{\ensuremath{\mathcal{F}}}
\newcommand{\pw}{\ensuremath{\mathrm{pw}}}
\newcommand{\BDM}{\ensuremath{\mathrm{BDM}}}
\newcommand{\RT}{\ensuremath{\mathrm{RT}}}
\newcommand{\osc}{\ensuremath{\mathrm{osc}}}
\renewcommand{\d}[1]{\ensuremath{\,\mathrm{d}#1}}
\renewcommand{\div}{\ensuremath{\operatorname{div}}}
\newcommand{\I}{\operatorname{I}}

\newcounter{cntS}
\newcommand{\newcnstS}{%
	\refstepcounter{cntS}%
	\ensuremath{c_{\thecntS}}}
\newcommand{\cnstS}[1]{\ensuremath{c_{\ref{#1}}}}

\newcounter{cntL}
\newcommand{\newcnstL}{%
	\refstepcounter{cntL}%
	\ensuremath{C_{\thecntL}}}
\newcommand{\cnstL}[1]{\ensuremath{C_{\ref{#1}}}}
\ifpdf
\hypersetup{
  pdftitle={},
  pdfauthor={}
}
\fi

\begin{document}
\maketitle

% REQUIRED
\begin{abstract}
The relaxation in the calculus of variation motivates the numerical analysis of a class of degenerate convex minimization 
problems with non-strictly convex energy densities with some convexity control and two-sided $p$-growth. % for $1 < p,p' < \infty$ with $1/p + 1/p' = 1$.
The minimizers may be non-unique in the primal variable but lead to a unique stress $\sigma \in H(\div,\Omega;\M)$.  Examples include the p-Laplacian, an optimal design problem in topology optimization, and the convexified double-well problem. 
The approximation by hybrid high-order methods (HHO) utilizes a reconstruction of the gradients with piecewise Raviart-Thomas or BDM finite elements without stabilization on a regular triangulation into simplices. 
The application of this HHO method to the class of degenerate convex minimization problems allows for a unique $H(\div)$ 
conforming stress approximation $\sigma_h$. The main results are a~priori 
and a posteriori error estimates for the stress error  $\sigma-\sigma_h$ in Lebesgue norms and a computable lower energy bound. Numerical
benchmarks display higher convergence rates for higher polynomial degrees and include adaptive mesh-refining with the first 
superlinear convergence rates of guaranteed lower energy bounds.
\end{abstract}

% REQUIRED

% REQUIRED

\section{Introduction}
This paper analyses the methodology \cite{DiPietroErn2015,DiPietroErnLemaire2014} for a class of degenerate convex minimization problems defined in \Cref{sec:energyDensity} with examples in \Cref{sec:examples}. The main results follow in \Cref{subsec:main}.
\subsection{A class of degenerate minimization problems and its finite element approximation}
The relaxation procedure in the calculus of variations \cite{Dacorogna2008} applies to minimization problems with non-convex energies and enforced microstructures \cite{BallJames1987} and provides an upscaling to a macroscopic model with a quasi-convexified energy density. In some model problems in nonlinear elasticity, multi-well problems, and topology optimization, the resulting energy density $W \in C^1(\M)$ with $\M = \R^{m \times n}$ is degenerate convex with a convexity control plus a two-sided growth of order $1< p < \infty$. Given a right-hand side $f \in L^{p'}(\Omega;\R^m)$ for $1/p+1/p'=1$ in a bounded polyhedral Lipschitz domain $\Omega \subset \R^n$, the minimal energy
\begin{align}
E(v) = \int_{\Omega} W(\D v) \d{x} - \int_{\Omega} f\cdot v \d{x} \quad\text{amongst } v\in V \coloneqq W^{1,p}_0(\Omega;\R^m) \label{intro:energy}
\end{align}
is attained, but the convex set of minimizers is not a singleton in general.
Nevertheless, the convexity control leads to a unique stress $\sigma = \D W(\D u) \in W^{1,p'}_{\mathrm{loc}}(\Omega;\M)$. A~priori and a~posteriori error estimates for the stress approximation are derived in \cite{CPlechac1997} for the lowest-order conforming scheme, followed by an adaptive scheme with plain convergence in \cite{BartelsC2008,CDolzmann2015}.
% The split of the energy density into a convex part and a null-Lagrangian extends this convergence result to the non-convex relaxed two-well problem in \cite{CDolzmann2015}.

The local stress regularity in \cite{CMueller02} motivated the mixed finite element approximation in \cite{CGuentherRabus2012}.
The better approximation of the stress variable through
Raviart-Thomas FEM on the one hand meets the non-smoothness of the dual functional $W^*$ on the other. The presence of a microstructure zone,
where the solution $u$ has a non-trivial gradient that generates a Young measure in the non-convex original problem \cite{BallKirchheimKristensen2000}, causes the so-called reliability-efficiency gap \cite{CJochimsen2003}: Efficient error estimates are not reliable and reliable error estimates are not efficient.
A one-point quadrature rule in the dual mixed
Raviart-Thomas formulation leads to the discrete Raviart-Thomas FEM in \cite{CLiu2015}, which is equivalent to a Crouzeix-Raviart FEM without a discrete duality gap. This allows the first guaranteed energy bounds
and an optimal a~posteriori error estimate, which overcomes the reliability efficiency gap in numerical examples for the optimal design problem. The lower energy bound (LEB) in \cite{CLiu2015} is restricted to the lowest-order discretization and higher-order schemes are not addressed in the literature.
Recent skeletal methods have been established in nonlinear problems \cite{DiPietroDroniou2017,AbbasErnPignet2018} with convergence rates in \cite{DiPietroDroniou2017-II,DiPietroDroniouManzini2018} for Leray–Lions problems and lead to lower eigenvalue bounds in \cite{CarstensenZhaiZhang19}.
%Benchmarks for three applications with parameters displayed in \Cref{tab:introExamples} test the performance of the new numerical scheme.
\subsection{A class of degenerate convex energy densities}\label{sec:energyDensity}
Suppose that the energy density $W \in C^1(\M)$ with $\M = \R^{m \times n}$ for $m,n \in \mathbb{N}$ satisfies the two-sided growth \eqref{ineq:growthW} and the convexity control \eqref{ineq:cc} with parameters $1 < p,p',r < \infty$, $0 \leq s < \infty$, $1/p + 1/p' = 1$:
There exist positive constants $\newcnstS\label{cnst:growthWLeft1}, \newcnstS\label{cnst:growthWRight1}, \newcnstS\label{cnst:cc}$ and non-negative constants $\newcnstS\label{cnst:growthWLeft2}, \newcnstS\label{cnst:growthWRight2}$ such that, for any $A,B \in \M$,
\begin{align}
\cnstS{cnst:growthWLeft1}|A|^p - \cnstS{cnst:growthWLeft2} &\leq W(A) \leq \cnstS{cnst:growthWRight1}|A|^p + \cnstS{cnst:growthWRight2}, \label{ineq:growthW}\\
\begin{split}
	|\D W(A) - \D W(B)|^r &\leq \cnstS{cnst:cc}(1 + |A|^s + |B|^s)\\
	&\quad\times(W(B) - W(A) - \D W(A):(B-A)). \label{ineq:cc}
\end{split}
\end{align}
\subsection{Examples}\label{sec:examples}
The following scalar examples with $m = 1$ displayed in \Cref{tab:introExamples} will be revisited in computational benchmarks in \Cref{sec:numericalExamples}. Further examples are found in \cite{CMueller02,Knees2008} and include Hencky elastoplasticity with hardening, vectorial two-well problems, and a special case of the Ericksen–James energy.
\begin{table}[H]
	\centering
	\begin{tabular}{|c|c|c|c|c|c|c|c|c|}
		\hline
		Examples & p & r & s & \cnstS{cnst:growthWLeft1} & \cnstS{cnst:growthWRight1} & \cnstS{cnst:cc} & \cnstS{cnst:growthWLeft2} & \cnstS{cnst:growthWRight2}\\
		\hline
		p-Laplacian & $2 \leq p < \infty$ & $2$ & $p - 2$ & $1/p$ & $1/p$ & $p$ & $0$ & $0$\\ 
		& $1 < p \leq 2$ & $p'$ & $0$ & & & & & \\
		\hline
		optimal design & $2$ & $2$ & $0$ & $\mu_1/2$ & $\mu_2/2$ & $2\mu_2$ & 0 & 0\\
		\hline
		relaxed double-well & $4$ & $2$ & $2$ & $1/8$ & $8$ & $\lambda$ & $\kappa$ & $\kappa$\\
		\hline
	\end{tabular}
	\vspace{0.5em}
	\caption{Parameters in \eqref{ineq:growthW}--\eqref{ineq:cc} in the examples of \Cref{sec:examples}}
	\label{tab:introExamples}
\end{table}
\vspace*{-1.5em}
\subsubsection{p-Laplace}\label{ex:pLaplace}
The minimization of the energy $E:V \to \R$ with the energy density $W:\R^n \to \R$ with $1 < p < \infty$ and
\begin{align*}
W(a) \coloneqq |a|^p/p \quad\text{for any } a\in \R^n
\end{align*}
is related to the nonlinear PDE $-\div\big(|\D v|^{p-2}\D v\big) = f \in L^{p'}(\Omega)$. The energy density $W$ satisfies \eqref{ineq:growthW}--\eqref{ineq:cc} with the constants of \Cref{tab:introExamples}. The value $\cnstS{cnst:cc} = 1 + \max\{1,p-2\}^2$ for $2 \leq p$ is derived in \cite[Lemma 2.2--2.3]{CKlose2003} for the related formula \eqref{ineq:ccMonotonicty}. The authors verified $\cnstS{cnst:cc} = p$ in \eqref{ineq:cc} for $p = 2,3,\dots,6$, but $\cnstS{cnst:cc} > p$ for integer $p = 7, 8, \dots$. 
\subsubsection{Optimal design problem}\label{ex:ODP}
The optimal design problem seeks the optimal distribution of two materials with fixed amounts to fill a given domain for maximal torsion stiffness \cite{KohnStrang1986,BartelsC2008}. For fixed parameters $0 < \xi_1 < \xi_2$ and $0 < \mu_1 < \mu_2$ with $\xi_1\mu_2 = \xi_2\mu_1$, the energy density $W(a) \coloneqq \psi(\xi)$, $a \in \R^n$, $\xi \coloneqq |a| \geq 0$ with
\begin{align*}
\psi(\xi) \coloneqq \begin{cases}
\mu_2 \xi^2/2 &\mbox{if } 0 \leq \xi \leq \xi_1,\\
\xi_1\mu_2(\xi - \xi_1/2) &\mbox{if } \xi_1 \leq \xi \leq \xi_2,\\
\mu_1 \xi^2/2 - \xi_1\mu_2(\xi_1/2 - \xi_2/2) &\mbox{if } \xi_2 \leq \xi
\end{cases}
\end{align*}
satisfies \eqref{ineq:growthW}--\eqref{ineq:cc} with the constants from \cite[Prop.~4.2]{BartelsC2008} displayed in \Cref{tab:introExamples}.
\subsubsection{Relaxed two-well problem}
The convex envelope $W$ of $|F - F_1|^2|F - F_2|^2$ for $F \in \R^n$ and fixed distinct $F_1, F_2 \in \R^n$ in the two-well problem of \cite{ChipotCollins1992} reads
\begin{align*}
W(F) = \max\{0,|F - B|^2 - |A|^2\}^2 + 4\big(|A|^2 |F-B|^2 - (A \cdot (F - B))^2 \big)
\end{align*}
with $A = (F_2 - F_1)/2$, $B = (F_1 + F_2)/2$, and satisfies \eqref{ineq:growthW}--\eqref{ineq:cc} with the constants of \Cref{tab:introExamples}, $\kappa \coloneqq 8\max\{|F_1|^4,|F_2|^4\}$ from \cite{CPlechac1997}, and $\cnstS{cnst:cc} = \lambda \coloneqq 32\max\{1, |A|^2, |A|^2/2 + 2|B|^2\}$ from \cite{C2008}.
%\subsubsection{An example of the Ericksen-James energy}
%Given $\kappa_1, \kappa_2 > 0$ and the bilinear form
%\begin{align*}
%	a(A,B) = a_{11}b_{12} + a_{12}b_{11} + a_{21}b_{22} + a_{22}b_{21} \quad\text{for } A,B \in \R^{2 \times 2},
%\end{align*}
%the convex envelope $W$ of $\kappa_1(|A|^2 - 2)^2 + 2^{-1}\kappa_2a(A,A)^2$ with
%\begin{align*}
%	W(A) = \begin{cases}
%		0 &\mbox{if } |a(A,A)| \leq 2-|A|^2,\\
%		\kappa_1(|A|^2 - 2)^2 + \frac{\kappa_2}{2}a(A,A)^2 &\mbox{if } \kappa_2 |a(A,A)| \leq 4\kappa_1(|A|^2 - 2),\\
%		\frac{\kappa_1\kappa_2}{4\kappa_1 + \kappa_2}(|A|^2-2+|a(A,A)|^2) &\mbox{otherwise},
%	\end{cases}
%\end{align*}
%satisfies \eqref{ineq:growthW} and \eqref{ineq:cc} with the parameters $p=4$,$r=2$, and $s=2$ \cite[Lemma 3.3]{Knees2008}
\subsection{Main results}\label{subsec:main}
The hybrid high-order (HHO) discretization features a split of the degrees of freedom into volume variables of polynomial degree at most $\ell$ and skeletal variables of polynomial degree at most $k$, $v_h = (v_\Tcal,v_\Fcal) \in V_h$.
The proposed numerical scheme replaces $\D v$ in \eqref{intro:energy} by a gradient reconstruction $R v_h$ of $v_h \in V_h$ in a linear space $\Sigma(\Tcal)$, the piecewise Raviart-Thomas or BDM finite element functions, for a shape-regular triangulation $\Tcal$ of $\Omega$ into simplices with maximal mesh-size $h_{\max}$.
This ensures the stability of $R$ in \Cref{lem:bondednessStabilityGradRec}.b, no additional penalization (called stabilization in HHO context) is required.
The discrete analog to \eqref{intro:energy} reads
\begin{align}
E_h(v_h) \coloneqq \int_{\Omega} W(R v_h) \d{x} - \int_\Omega f\cdot v_\Tcal \d{x} \quad\text{for } v_h = (v_\Tcal,v_\Fcal) \in V_h. \label{pr:disMinEnergy}
\end{align}
Details on the HHO method and the linear map $R:V_h \to \Sigma(\Tcal)$ follow in \Cref{sec:HHOMethod} and include the proofs of the following statements.
Any discrete minimizer $u_h \in V_h$ defines the unique discrete stress $\sigma_h \coloneqq \Pi_{\Sigma(\Tcal)} \D W(R u_h) \in \Sigma(\Tcal)$ with the $L^2$ projection $\Pi_{\Sigma(\Tcal)}$ onto $\Sigma(\Tcal)$.
The results from \Cref{sec:error} apply to the examples in \Cref{tab:introExamples} and lead to the a~priori and a~posteriori estimates in \Cref{thm:apriori}--\ref{thm:aposteriori}, and
imply the convergence rate $\|\sigma - \sigma_h\|_{L^{p'}(\Omega)} + |E(u) - E_h(u_h)| \lesssim h_{\max}^{(k+1)/r}$ (\Cref{cor:convergenceRate}) for smooth functions $\sigma,u,f$.
This extends the a~priori results in \cite{CGuentherRabus2012} to methods of higher polynomial degrees. Let $u$ be an arbitrary minimizer of $E$ in $V$ and $\sigma \coloneqq \D W(\D u)$.
\begin{theorem}[a~priori]\label{thm:apriori}
	There exist positive constants $\cnstL{cnst:aprioriLeft2}, \dots, \cnstL{cnst:aprioriRight}$ such that any discrete minimizer $u_h$ of $E_h$ in $V_h$ and the discrete stress $\sigma_h \coloneqq \Pi_{\Sigma(\Tcal)} \D W(R u_h)$ satisfy (a)--(c).
	\begin{enumerate}[wide]
		\item[(a)] The discrete stress $\sigma_h$ is unique in the sense that the definition does not depend on the choice of the (possibly non-unique) discrete minimizer $u_h$.
		\item[(b)]
		$\sigma_h \in Q(f,\Tcal) \coloneqq \{\tau_h \in \Sigma(\Tcal) \cap W^{p'}(\div,\Omega;\M):\div \tau_h = - \Pi_\Tcal^\ell f\}$.
		\item[(c)]
		$\begin{aligned}[t]
		&\max\big\{\newcnstL\label{cnst:aprioriLeft2}^{-1}\|\sigma - \sigma_h\|^r_{L^{p'}(\Omega)},  \newcnstL\label{cnst:aprioriLeft1}^{-1}\|\sigma - \D W(R u_h)\|^r_{L^{p'}(\Omega)}\big\}\\
		&\quad \leq E^*(\sigma) - \max E^*(Q(f,\Tcal)) + \newcnstL\label{cnst:DuCP} \osc_\ell(f,\Tcal) + \newcnstL\label{cnst:aprioriRight} \|(1 - \Pi_{\Sigma(\Tcal)})\D u\|^{r'}_{L^{p}(\Omega)}
		\end{aligned}$\\
		with the dual energy $E^*$ of $E$ from \eqref{def:dualEnergy} below.
	\end{enumerate}
\end{theorem}
The guaranteed lower energy bound $\mathrm{LEB} \coloneqq E^*(\sigma_h) - \cnstL{cnst:DuCP}\mathrm{osc}_\ell(f,\mathcal{T}) \leq \min E(V)$ in \Cref{thm:aposteriori}.a below displays superlinear convergence rates. For the lowest-order discretization $k = 0$ and $\ell = 0$, this is superior in comparison to \cite{Ortner2011,OrtnerPraetorius2011} in the sense that $\|h f\|_{L^{p'}(\Omega)}$ is replaced by the higher-order term $\osc_0(f,\Tcal)$. The a~posteriori estimate in \Cref{thm:aposteriori}.b enables guaranteed error control without additional information and motivates an adaptive scheme.
\begin{theorem}[a~posteriori]\label{thm:aposteriori}
	Let $u_h$ minimize $E_h$ in $V_h$.
	There exist positive constants $\cnstL{cnst:LEB},\dots,\cnstL{cnst:aposterioriRight}$ such that the discrete stress $\sigma_h \coloneqq \Pi_{\Sigma(\Tcal)} \D W(R u_h)$ and any $v \in W^{1,p}_0(\Omega;\R^m)$ satisfy (a)--(b).
	\begin{enumerate}[wide]
		\item[(a)] (LEB) $\newcnstL\label{cnst:LEB}^{-1}\|\sigma - \sigma_h\|^r_{L^{p'}(\Omega)} + E^*(\sigma_h) - \cnstL{cnst:DuCP} \osc_\ell(f,\Tcal) \leq \min E(V)$.
		\item[(b)]$\begin{aligned}[t]
		&\newcnstL\label{cnst:aposterioriLeft1}^{-1}\|\sigma - \sigma_h\|^r_{L^{p'}(\Omega)} + \cnstL{cnst:aprioriLeft1}^{-1}\|\sigma - \D W(R u_h)\|^r_{L^{p'}(\Omega)} \leq E_{h}(u_h) - E^*(\sigma_h)\\
		&\qquad\qquad + \cnstL{cnst:DuCP} \textup{osc}_\ell(f,\mathcal{T}) + \newcnstL\label{cnst:aposterioriRight} \|R u_h - \D v\|_{L^{p}(\Omega)}^{r'} - \int_\Omega v\cdot(1-\Pi_\Tcal^\ell)f \d{x}.
		\end{aligned}$
	\end{enumerate}
\end{theorem}
Notice that the right-hand side of \Cref{thm:aposteriori}.b is computable with some post-processing of $v \in W^{1,p}_0(\Omega;\R^m)$ as demonstrated in the numerical examples of \Cref{sec:numericalExamples}.
\subsection{Outline of this paper}
The remaining parts of this paper are organized as follows. \Cref{sec:continuousSetting} recalls known results on the continuous level. \Cref{sec:HHOMethod} reviews the discretization with the HHO methodology. The error analysis in \Cref{sec:error} is established in a general framework.
The a~priori results in \Cref{sec:apriori} include error estimates for the stress approximation $\sigma - \sigma_{h}$ and the energy difference $E(u) - E_h(u_{h})$ as well as a discussion on the convergence rates.
A lower energy bound of $\min E(V)$ is the point of departure in the a~posteriori analysis in \Cref{sec:aposteriori}. The a~posteriori estimates allow the computation of a guaranteed upper error bound with some post-processing.
Numerical results for the three model examples of \Cref{tab:introExamples} are presented in \Cref{sec:numericalExamples} with conclusions drawn from the numerical experiments.
\subsection{Notation}
Standard notations for Sobolev and Lebesgue functions and in convex analysis (recalled in \Cref{sec:preliminary} below) apply throughout this paper. In particular, $(\bullet, \bullet)_{L^2(\Omega)}$ denotes the scalar product of $L^2(\Omega)$ and $W^{p'}(\div,\Omega;\M) \coloneqq W^{p'}(\div,\Omega)^m$ is the matrix-valued version of
\begin{align}
W^{p'}(\div,\Omega) \coloneqq \{\tau \in L^{p'}(\Omega;\R^n): \div \tau \in L^{p'}(\Omega)\}.\label{def:Hdiv}
\end{align}
For a Banach space $X$, $\mathcal{L}(X)$ denotes the space of bounded linear operators $L:X \to X$ endowed with the operator norm $\|\bullet\|_{\mathcal{L}(X)}$. 
For any $A,B \in \M \coloneqq \R^{m \times n}$, $A:B$ denotes the Euclidean scalar product of $A$ and $B$, which induces the Frobenius norm $|A| \coloneqq (A:A)^{1/2}$ in $\M$. For $1 < p < \infty$, $p' = p/(p-1)$ denotes the H\"older conjugate of $p$ with $1/p + 1/p' = 1$.
The notation $A \lesssim B$ abbreviates $A \leq CB$ for a generic constant $C$ independent of the mesh-size and $A \approx B$ abbreviates $A \lesssim B \lesssim A$. Generic constants are written as $c_j$ or $C_j$, where $\cnstS{cnst:growthWLeft1}, \dots, \cnstS{cnst:growthDWStarCnst}$ exclusively depend on $\cnstS{cnst:growthWLeft1},\dots,\cnstS{cnst:growthWRight2}$, while $\cnstL{cnst:aprioriLeft2}, \dots, \cnstL{cnst:xi}$ may additionally depend on the domain, the shape-regularity of the triangulations, and the parameters $k,\ell,m,n,p,r,s,t$ ($t$ from \Cref{sec:error}).
\section{Convex analysis and preliminaries on the continuous level}\label{sec:continuousSetting}
This section starts with common duality tools in convex analysis and further properties of the energy density $W$, followed by a summary of known results concerning the minimizer $u$ and the stress $\sigma \coloneqq \D W(\D u)$. 
\subsection{Convex analysis for functions with two-sided growth}\label{sec:preliminary}
Let $W^*:\M \to \R$ denote the convex conjugate of $W$ \cite[Corollary 12.2.2]{Rockafellar1970} with
\begin{align}
W^*(G) \coloneqq \sup_{A \in \R^{m \times n}} (G:A - W(A)) \quad \text{for any } G \in \M.\label{def:ConvexConj}
\end{align}
The subdifferential $\partial W^*$ of $W^*$ at $G \in \M$ \cite[Section 23]{Rockafellar1970} is the set of matrices $A \in \M$ that satisfy
\begin{align}
A:(H - G) \leq W(H) - W(G) \quad\text{for all } H \in \M.\label{def:SubDiff}
\end{align}
The relation $A \in \partial W^*(G)$ is equivalent to $G:A = W(A) + W^*(G)$ \cite[Theorem 23.5]{Rockafellar1970}. This implies the equivalence of the convexity control in \eqref{ineq:cc} and
\begin{align}
|G - H|^r &\leq \cnstS{cnst:cc}(1 + |A|^s + |B|^s)(W^*(H) - W^*(G) - A:(H-G)) \label{ineq:ccDual}
\end{align}
for any $G,H \in \M$, $A \in \partial W^*(G)$, and $B \in \partial W^*(H)$.
\begin{lemma}\label{lem:GrowthConj}
	Let $W \in C^1(\M)$ be convex and satisfy the two-sided growth in \eqref{ineq:growthW}. Then $W^* \in C(\M)$ and there exist constants $\cnstS{cnst:growthDW},\dots,\cnstS{cnst:growthDWStarCnst}$ that satisfy (a)--(c).
	\begin{enumerate}[wide]
		\item[(a)] (growth of $\D W$) $|\D W(A)|^{p'} \leq \newcnstS\label{cnst:growthDW} |A|^p + \newcnstS\label{cnst:growthDWCnst}$ for all $A \in \M$.
		\item[(b)] (two-sided growth of $W^*$) $\newcnstS\label{cnst:growthWStarLeft1} |G|^{p'} - \cnstS{cnst:growthWRight2} \leq W^*(G) \leq \newcnstS\label{cnst:growthWStarRight1} |G|^{p'} + \cnstS{cnst:growthWLeft2}$ for all $G \in \M$.
		\item[(c)] (growth of $\partial W^*$) $|A|^{p} \leq \newcnstS\label{cnst:growthDWStar} |G|^{p'} + \newcnstS\label{cnst:growthDWStarCnst}$ for all $G \in \M$ and $A \in \partial W^*(G)$.
	\end{enumerate}
\end{lemma}
\begin{proof}
	The proofs involve elementary calculations only and are outlined below.
	
	\emph{Proof of \Cref{lem:GrowthConj}.b.} The growth condition in \eqref{ineq:growthW} and the definition of $W^*$ in \eqref{def:ConvexConj} imply, for any $A,G \in \M$, that
	\begin{align}
	G:A - \cnstS{cnst:growthWRight1}|A|^{p} - \cnstS{cnst:growthWRight2} \leq G:A - W(A) \leq G:A - \cnstS{cnst:growthWLeft1}|A|^p + \cnstS{cnst:growthWLeft2}. \label{ineq:PrTSGr}
	\end{align}
	The choice $A \coloneqq (\cnstS{cnst:growthWRight1}p)^{1-p'}|G|^{(2-p)/(p-1)} G$ in \eqref{ineq:PrTSGr} results in
	\begin{align*}
	G:A - \cnstS{cnst:growthWRight1}|A|^{p} - \cnstS{cnst:growthWRight2} = (\cnstS{cnst:growthWRight1}p)^{1-p'}(p')^{-1}|G|^{p'} - \cnstS{cnst:growthWRight2} \eqqcolon \cnstS{cnst:growthWStarLeft1}|G|^{p'} - \cnstS{cnst:growthWRight2}.
	\end{align*}
	This proves $\cnstS{cnst:growthWStarLeft1}|G|^{p'} - \cnstS{cnst:growthWRight2} \leq W^*(G)$ for all $G \in \M$. H\"older and Young inequality show that
	$G:A \leq (\cnstS{cnst:growthWLeft1} p)^{1-p'}(p')^{-1}|G|^{p'} + \cnstS{cnst:growthWLeft1}|A|^p \eqqcolon \cnstS{cnst:growthWStarRight1}|G|^{p'} + \cnstS{cnst:growthWLeft1}|A|^p$.
	This and \eqref{ineq:PrTSGr} imply $W^*(G) \leq \cnstS{cnst:growthWStarRight1}|G|^{p'} + \cnstS{cnst:growthWLeft2}$ for all $G \in \M$.
	
	\emph{Proof of \Cref{lem:GrowthConj}.c.} The choice $H \coloneqq G + 2^{-1}(\cnstS{cnst:growthWStarRight1}p')^{1-p}|A|^{p-2}A$ in \eqref{def:SubDiff} shows
	\begin{align*}
	A:(H - G) = 2^{-1}(\cnstS{cnst:growthWStarRight1}p')^{1-p}|A|^{p}.
	\end{align*}
	The two-sided growth of $W^*$ from (b) and Jensen's inequality prove that
	\begin{align*}
	W^*(H) - W^*(G) &\leq \cnstS{cnst:growthWStarRight1}|H|^{p'} + \cnstS{cnst:growthWLeft2} - \cnstS{cnst:growthWStarLeft1}|G|^{p'} + \cnstS{cnst:growthWRight2}\\
	&\leq (2^{p'-1}\cnstS{cnst:growthWStarRight1} - \cnstS{cnst:growthWStarLeft1})|G|^{p'} + \cnstS{cnst:growthWLeft2} + \cnstS{cnst:growthWRight2} + 2^{-1}\cnstS{cnst:growthWStarRight1}^{1-p}(p')^{-p}|A|^{p}.
	\end{align*}
	Since $A:(H - G) \leq W^*(H) - W^*(G)$, the combination of the two previous formulas results in $2^{-1}\cnstS{cnst:growthWStarRight1}^{1-p}(p')^{-p}(p'-1)|A|^{p} \leq (2^{p'-1}\cnstS{cnst:growthWStarRight1} - \cnstS{cnst:growthWStarLeft1})|G|^{p'} + \cnstS{cnst:growthWLeft2} + \cnstS{cnst:growthWRight2}$ for all $G \in \M$. This proves (c).
	
	\emph{Proof of \Cref{lem:GrowthConj}.a.} Note that the proof of (c) only requires a two-sided growth of $W^*$ of order $p'$. The same arguments apply to $W$ and show $|\D W(A)|^{p'} \leq \cnstS{cnst:growthDW}|A|^{p} + \cnstS{cnst:growthDWCnst}$.
\end{proof}
\begin{remark}[monotonicity]\label{rem:monotonicity}
	The convexity control in \eqref{ineq:cc} implies the monotonicity of $\D W$ in the sense that, for any $A,B \in \M$,
	\begin{align}
	\begin{split}
		|\D W(A) - \D W(B)|^r &\leq \cnstS{cnst:cc}(1 + |A|^s + |B|^s)\\
		&\qquad\times(\D W(A) - \D W(B)):(A - B).
	\end{split}\label{ineq:ccMonotonicty}
	\end{align}
	Conversely, the monotonicity of $\D W$, the growth of $\D W$ in \Cref{lem:GrowthConj}.a, and the growth of $\partial W^*$ in \Cref{lem:GrowthConj}.c imply the convexity control in \eqref{ineq:cc} \cite[Lemma 2.2]{Knees2008}.
\end{remark}
\subsection{A review of known results}\label{sec:knownResults}
Recall the continuous energy $E$ from \eqref{intro:energy} with \eqref{ineq:growthW}--\eqref{ineq:cc} and $W^*$ from \eqref{def:ConvexConj}. Define the dual energy $E^*:L^{p'}(\Omega;\M) \to \R$ by
\begin{align}
E^*(\tau) \coloneqq -\int_{\Omega} W^*(\tau) \d{x} \quad\text{for } \tau \in L^{p'}(\Omega;\M)\label{def:dualEnergy}
\end{align}
and recall $W^{p'}(\div,\Omega;\M)$ from \eqref{def:Hdiv}.
\begin{theorem}\label{thm:reviewResults} The minimal energy $\min E(V)$ is attained. Any minimizer $u$ of $E$ in $V$ and
the stress $\sigma \coloneqq \D W(\D u) \in L^{p'}(\Omega;\M)$ satisfy (a)--(d) with constants $\cnstL{cnst:Du}$ and $\cnstL{cnst:sigma}$.
	\begin{enumerate}[wide]
		\item[(a)] The definition of $\sigma$ does not depend on the choice of the minimizer $u$ and $\sigma \in Q(f) \coloneqq \{\tau\in W^{p'}(\div,\Omega;\M): \div \sigma = -f\}$.
		\item[(b)] The stress $\sigma$ is the unique maximizer of the dual energy $E^*$ in $Q(f)$
		without duality gap in the sense that $\min E(V) = E(u) = E^*(\sigma) = \max E^*(Q(f))$.
		\item[(c)] $\|\D u\|_{L^p(\Omega)} \leq \newcnstL\label{cnst:Du}$ and $\|\sigma\|_{L^{p'}(\Omega)} \leq \newcnstL\label{cnst:sigma}$.
		\item[(d)] If $f \in L^{p'}(\Omega;\R^m) \cap W^{1,p'}_{\mathrm{loc}}(\Omega;\R^m)$, then $\sigma \in W^{1,p'}_{\mathrm{loc}}(\Omega;\M)$.
		\item[(e)] Suppose that $p+s < pr$, $q \coloneqq pr/(p+s) > 1$, $\delta > 0$, and $f \in L^{p'}(\Omega;\R^m)$. Then $\sigma \in W^{1/r - \delta,q}(\Omega;\M)$.
		If $q/r < n$ (resp.~$q/r = n$), then $\sigma \in L^{q^*}(\Omega;\M)$ for $q^* < nq/(n-q/r)$ (resp.~$q^* = \infty$).
		In particular, the choice $q = p'$ in any example of \Cref{sec:examples} leads to $\sigma \in W^{1/r-\delta, p'}(\Omega;\M)$.
	\end{enumerate}
\end{theorem}
\begin{proof}
	The existence of a minimizer follows from the direct method of the calculus of variations \cite[Theorem 3.30]{Dacorogna2008}.
	The growth of $\D W$ in \Cref{lem:GrowthConj}.a leads to $\sigma \in L^{p'}(\Omega;\M)$ and \cite[Theorem 3.37]{Dacorogna2008} guarantees the Euler-Lagrange equations
	\begin{align}
	\int_{\Omega} \sigma:\D v \d{x} = \int_{\Omega} f\cdot v\d{x} \quad\text{for any } v \in V.\label{eq:ELE}
	\end{align}
	
	\emph{Proof of (a).} The uniqueness of $\sigma$ is shown in \cite[Theorem 2]{CPlechac1997}. The Euler-Lagrange equations \eqref{eq:ELE} imply $f + \div \sigma = 0$ and $\sigma \in Q(f)$.
		
	\emph{Proof of (b).} The definition of $W^*$ in \eqref{def:ConvexConj} proves $\tau:\D u \leq W(\D u) + W^*(\tau)$ a.e.~in $\Omega$ for any $\tau \in Q(f)$. An integration by parts shows that
	\begin{align}
	E^*(\tau) = -\int_{\Omega} W^*(\tau) \d{x} \leq \int_\Omega W(\D u) \d{x} - \int_{\Omega} \tau:\D u \d{x} = E(u).\label{ineq:dualEnergy}
	\end{align}
	The duality $\D u \in \partial W^*(\sigma)$ reads $\sigma:\D u = W(\D u) + W^*(\sigma)$ a.e.~in $\Omega$ \cite[Corollary 12.2.2]{Rockafellar1970}. The combination of \eqref{eq:ELE}--\eqref{ineq:dualEnergy} leads to $E^*(\tau) \leq E^*(\sigma) = E(u)$. Notice from \eqref{ineq:ccDual} that $W^*$ is strictly convex. Hence, the maximizer of $E^*$ in $Q(f)$ is unique.
	
	\emph{Proof of (c).} The constant $\cnstL{cnst:Du}$ in (c) is the positive root of the function $\cnstS{cnst:growthWLeft1}x^p -  C_P\|f\|_{L^{p'}(\Omega)}x - \cnstS{cnst:growthWLeft2}|\Omega| - E(0)$ \cite[Proof of Theorem 2]{CPlechac1997} in $x > 0$ with the Poincar\'e constant $C_P$. On convex domains, $C_P \leq 1/\pi$ is proven in \cite{Bebendorf2003,PayneWeinberger1960} for $p = 2$.
	The growth of $\D W$ in \Cref{lem:GrowthConj}.a leads to
	$\|\sigma\|_{L^{p'}(\Omega)}^{p'} \leq \cnstS{cnst:growthDW}\cnstL{cnst:Du}^p + \cnstS{cnst:growthDWCnst}|\Omega| \eqqcolon \cnstL{cnst:sigma}^{p'}$.
	
	\emph{Proof of (d).} The local stress regularity $\sigma \in W^{1,p'}_{\mathrm{loc}}(\Omega;\M)$ is shown for $f \in L^{p'}(\Omega;\R^m) \cap W^{1,p'}_{\mathrm{loc}}(\Omega;\R^m)$ in \cite[Theorem 2.1]{CMueller02}.
	
	\emph{Proof of (e).} The global stress regularity $\sigma \in W^{1/r - \delta,q}(\Omega;\M)$ is shown in \cite[Theorem 2.2]{Knees2008}. This and the Sobolev embedding $W^{1/r-\delta,q}(\Omega;\M) \hookrightarrow L^{q*}(\Omega;\M)$ for any $q^* < nq/(n-q/r)$ if $q/r < n$ \cite[Theorem 6.7]{DiNezzaPalatucciValdinoci2012} and $q^* = \infty$ if $q/r = n$ \cite[Theorem 6.10]{DiNezzaPalatucciValdinoci2012} prove the assertion (e). 
\end{proof}
\section{HHO method}\label{sec:HHOMethod}
This section introduces discrete spaces, the gradient reconstruction, and the discrete problem.
\subsection{Triangulation}
A regular triangulation $\Tcal$ of $\Omega$ in the sense of Ciarlet is a finite set of closed simplices $T$ of positive volume $|T| > 0$ with boundary $\partial T$ and outer unit normal $\nu_T$ such that $\cup_{T \in \Tcal} T = \overline{\Omega}$ and two distinct simplices are either disjoint or share one common (lower-dimensional) subsimplex (vertex or edge in 2D and vertex, edge, or face in 3D). 
Let $\Fcal(T)$ denote the set of the $n+1$ hyperfaces of $T$, called sides of $T$, and define the set of all sides $\Fcal = \cup_{T \in \Tcal} \Fcal(T)$ and the set of interior sides $\Fcal(\Omega) = \Fcal\setminus\{F \in \Fcal:F\subset\partial \Omega\}$ in $\Tcal$.

For any interior side $F \in \Fcal(\Omega)$, there exist exactly two simplices $T_+, T_- \in \Tcal$ such that $\partial T_+ \cap \partial T_- = F$. The orientation of the outer normal unit $\nu_F = \nu_{T_+}|_F = -\nu_{T_-}|_F$ along $F$ is fixed. Define the side patch $\omega_F \coloneqq \mathrm{int}(T_+ \cup T_-)$ of $F$ and let $[v]_F \coloneqq (v|_{T_+})|_F - (v|_{T_-})|_F \in L^1(F)$ denote the jump of $v \in L^1(\omega_F)$ with $v \in W^{1,1}(T_+)$ and $v \in W^{1,1}(T_-)$ across $F$.
For any boundary side $F \in \Fcal(\partial \Omega) \coloneqq \Fcal\setminus\Fcal(\Omega)$, $\nu_F \coloneqq \nu_T$ is the exterior unit vector for $F \in \Fcal(T)$ with $T \in \Tcal$ and $[v]_F \coloneqq (v|_T)|_F$.
The characteristic function $\chi_T \in L^\infty(\Omega)$ of $T \in \Tcal$ is equal to $1$ in $T$ and vanishes elsewhere. 
The differential operators $\div_{\pw}$ and $\D_\pw$ depend on the triangulation $\Tcal$ and denote the piecewise application of $\div$ and $\D$ without explicit reference to the triangulation $\Tcal$.

\subsection{Discrete spaces}
The discrete ansatz space of the HHO methods consists of piecewise polynomials on the triangulation $\Tcal$ and on the skeleton $\partial \Tcal \coloneqq \cup \Fcal$. For a simplex or a side $M \subset \R^n$ of diameter $h_M$, let $P_\ell(M)$ denote the space of polynomials of maximal order $\ell$ regarded as functions defined in $M$.
The $L^2$ projection $\Pi_M^\ell v \in P_\ell(M)$ of $v \in L^1(M)$ satisfies
\begin{align*}
	\int_{M} (1-\Pi_M^{\ell})v p_\ell \d{x} = 0 \quad\text{for any } p_\ell\in P_\ell(M).
\end{align*}
The local mesh sizes give rise to the piecewise constant function $h_\Tcal \in P_0(\Tcal)$ with $h_\Tcal|_T \equiv h_T$ in $T \in \Tcal$. Let $\osc_\ell(f,\Tcal) \coloneqq \|h_\Tcal(1 - \Pi_{\Tcal}^\ell) f\|_{L^{p'}(\Omega)}$ denote the data oscillation of $f$ in $\Tcal$.
The gradient reconstruction in $T \in \Tcal$ maps in the piecewise Brezzi-Douglas-Marini finite element functions $\BDM_{k}(T) \coloneqq P_k(T;\R^n)$ or in the space of Raviart-Thomas functions
\begin{align*}
	\RT_k(T) &\coloneqq P_k(T;\R^n) + x P_k(T) \subset \BDM_{k+1}(T).
\end{align*}
Let $P_\ell(\Tcal)$, $P_k(\Fcal)$, $\RT^\pw_k(\Tcal)$, and $\BDM_{k}^\pw(\Tcal)$ denote the space of piecewise functions (with respect to $\Tcal$ and $\Fcal$) with restrictions to $T$ or $F$ in $P_\ell(T)$, $P_k(F)$, $\RT_k(T)$, and $\BDM_k(T)$. Let $\Pi_{\Tcal}^\ell$, $\Pi_{\Fcal}^k$, and $\Pi_{\RT^\pw_k(\Tcal)}$ denote the $L^2$ projections onto the respective discrete spaces. For vector-valued functions $v \in L^1(\Omega;\R^m)$, the $L^2$ projection $\Pi_{\Tcal}^\ell$ onto $P_\ell(\Tcal;\R^m) \coloneqq P_\ell(\Tcal)^m$ applies elementwise. This applies to the $L^2$ projections onto $P_\ell(M;\R^m)$, $P_k(\Fcal;\R^m) \coloneqq P_k(\Fcal)^m$, $\RT_k(\Tcal;\M) \coloneqq \RT_k(\Tcal)^m$, or $\BDM_{k}(\Tcal;\M) \coloneqq \BDM_{k}(\Tcal)^m$ etc.
\subsection{HHO ansatz space}
For fixed $k \in \mathbb{N}_0$ and non-negative $\ell \in \{k,k-1\}$, let
\begin{align}
	V_h \coloneqq P_\ell(\Tcal;\R^m) \times P_k(\Fcal(\Omega);\R^m)\label{def:discretePairing}
\end{align}
denote the discrete ansatz space of HHO methods \cite{DiPietroErnLemaire2014,DiPietroErn2015} with two examples in \eqref{def:discretePairing1} below.
The interior sides $\Fcal(\Omega)$ give rise to $P_k(\Fcal(\Omega);\R^m)$ as the subspace of all $(v_F)_{F \in \Fcal} \in P_k(\Fcal;\R^m)$ with the convention that $v_F = 0$ on any boundary side $F \in \Fcal(\partial \Omega)$ for homogenous boundary conditions. In other words, the notation $v_h \in V_h$ means that $v_h = (v_\Tcal,v_\Fcal) = \big((v_T)_{T \in \Tcal},(v_F)_{F \in \Fcal}\big)$ for some $v_\Tcal \in P_\ell(\Tcal;\R^m)$ and $v_\Fcal \in P_k(\Fcal(\Omega);\R^m)$ with the identification $v_T = v_\Tcal|_T \in P_\ell(T;\R^m)$ and $v_F = v_\Fcal|_F \in P_k(F;\R^m)$. The discrete norm $\|\bullet\|_h$ of $V_h$ from \cite{DiPietroDroniou2017} is defined, for any $v_h \in V_h$, by
\begin{align*}
\|v_h\|_{h}^p \coloneqq \sum_{T \in \Tcal} \|v_h\|_{h,T}^p \quad\text{and}\quad \|v_h\|_{h,T}^p \coloneqq \|\D v_T\|_{L^p(T)}^p + \sum_{F \in \Fcal(T)} h_F^{1-p}\|v_F - v_T\|^p_{L^p(F)}.
\end{align*}
The interpolation $\I : V \to V_h$ maps $v \in V$ onto $\I v \coloneqq (\Pi_{\Tcal}^\ell v, \Pi_{\Fcal}^k v) \in V_h$. Two examples for $V_h$ and $\Sigma(\Tcal)$ are in the focus of this work with
\begin{align}
	\begin{split}
		\Sigma(\Tcal) &= \RT_k^{\pw}(\Tcal;\M) \text{ and } \ell = k \text{ or}\\
		\Sigma(\Tcal) &= \BDM_{k}^{\pw}(\Tcal;\M) \text{ and } \ell = k-1 \in \mathbb{N}_0.
	\end{split}
	\label{def:discretePairing1}
\end{align}
Notice that $k = \ell \in \mathbb{N}_0$ or $k = \ell + 1 \in \mathbb{N}$ in \eqref{def:discretePairing1} and these two cases are labelled by $k = \ell$ and $k = \ell + 1$ throughout the paper; e.g., ``$k = \ell$ in \eqref{def:discretePairing1}'' means in particular that $\Sigma(\Tcal)$ is a Raviart-Thomas finite element space.
\subsection{Gradient reconstruction}
The gradient reconstruction $R:V_h \to \Sigma(\Tcal)$ in \cite{AbbasErnPignet2018} maps $v_h \in V_h$ onto $R v_h \in \Sigma(\Tcal)$ such that, for any $\tau_h \in \Sigma(\Tcal)$,
\begin{align}
\int_\Omega R v_h:\tau_h \d{x} &= -\int_{\Omega} v_\Tcal \cdot \div_{\pw} \tau_h \d{x} + \sum_{F \in \Fcal} \int_{F} v_F \cdot [\tau_h \nu_F]_F \d{s}
\label{def:grRecGeneral}
\end{align}
with the normal jump $[\tau_h \nu_F]_F$ of $\tau_h$ across $F$.
In other words, $R v_h$ is the Riesz representation of the linear functional on the right-hand side of \eqref{def:grRecGeneral} in the Hilbert space $\Sigma(\Tcal)$ endowed with the $L^2$ scalar product. Although $R$ is described here as a global operator, it acts locally and can be computed in parallel for each simplex.
\begin{lemma}\label{lem:bondednessStabilityGradRec}
	The gradient reconstruction operator $R$ satisfies (a)--(d) for any $v \in V$ and $v_h = (v_\Tcal,v_\Fcal) \in V_h$ with some $1 < p < \infty$ and the generic constant $C_\mathrm{dF}$, which depends on the shape regularity of $\Tcal$.
	\begin{enumerate}[wide]
		\item[(a)] (boundedness) $\|R v_h\|_{L^p(\Omega)} \lesssim \|v_h\|_{h}$.
		\item[(b)] (stability) $\|v_h\|_{h} \lesssim \|R v_h\|_{L^p(\Omega)}$.
		\item[(c)] (discrete Friedrichs inequality) $\|v_\Tcal\|_{L^p(\Omega)} \leq C_{\mathrm{dF}} \|R v_h\|_{L^p(\Omega)}$.
		\item[(d)] (commutativity) $\Pi_{\Sigma(\Tcal)} \D v = R \I v$ for the $L^2$ projection $\Pi_{\Sigma(\Tcal)}$ onto $\Sigma(\Tcal)$.
	\end{enumerate}
\end{lemma}
\begin{proof}
	The proof follows \cite[Lemma 1]{AbbasErnPignet2018}, where (a), (b), and (d) are established for $p = 2$. The extension to the case $p \neq 2$ is discussed briefly below.
	
	\emph{Proof of (a).} For any $T \in \Tcal$, the norm equivalence in finite-dimensional spaces shows that $\|R v_h\|_{L^p(T)}\|R v_h\|_{L^{p'}(T)} \approx \|R v_h\|_{L^2(T)}^2$.
	The choice $\tau_h = \chi_T R v_h$ in \eqref{def:grRecGeneral} and an integration by parts imply
	\begin{align*}
		\|R v_h\|_{L^2(T)}^2 = \int_T \D v_T : R v_h \d{x} + \sum_{F \in \Fcal(T)} \int_{F} (v_F - v_T)\cdot (Rv_h \nu_T|_F) \d{s}.
	\end{align*}
	The discrete trace inequality $\|(R v_h)|_T\|_{L^{p'}(F)} \lesssim h_F^{-1/p'}\|R v_h\|_{L^{p'}(T)}$ for any $F \in \Fcal(T)$ and the Cauchy inequality conclude the proof of (a).
	
	\emph{Proof of (b).} Let $k = \ell$.
	For any $T \in \Tcal$, define $\tau_h \in \RT_k(T;\M)$ via the moments
	\begin{align*}
	\Pi_T^{k-1} \tau_h &= \Pi_T^{k-1} \big(|\D v_T|^{p-2}\D v_T\big),\\
	\Pi_F^k (\tau_h \nu_T|_F) &= h_{F}^{1-p} \Pi_F^k \big(|v_{F} - v_T|^{p-2} (v_{F} - v_T)\big) \quad\text{for any } F \in \Fcal(T)
	\end{align*}
	with the convention $P_{-1}(T;\M) = \{0\}$.
	An integration by parts in \eqref{def:grRecGeneral} leads to 
	\begin{align}
	\|v_h\|_{h,T}^p = (R v_h,\tau_h)_{L^2(T)} \leq \|R v_h\|_{L^p(T)} \|\tau_h\|_{L^{p'}(T)}.\label{ineq:proofStabilityGradRec}
	\end{align}
	The stability of Raviart-Thomas functions in terms of their canonical degrees of freedom \cite[Proposition 2.3.4]{BoffiBrezziFortin2013} and a scaling argument prove that
	\begin{align*}
	\|\tau_h\|_{L^{p'}(T)}^{p'} &\approx \|\Pi_T^{k-1}(|\D v_T|^{p-2}\D v_T)\|_{L^{p'}(T)}^{p'}\\
	&\qquad + \sum_{F \in \Fcal(T)} h_F \|h_F^{1-p}\Pi_F^k(|v_{F} - v_T|^{p-2}(v_F - v_T))\|^{p'}_{L^{p'}(F)}.
	\end{align*}
	The stability of the $L^2$ projections $\Pi_T^{k-1}$ and $\Pi_F^k$ in the $L^{p'}$ norm \cite[Lemma 3.2]{DiPietroErn2015} leads to $\|\tau_h\|_{L^{p'}(T)}^{p'} \lesssim \|v_h\|_{h,T}^{p}$. This and \eqref{ineq:proofStabilityGradRec} imply $\|v_h\|_{h,T} \lesssim \|R v_h\|_{L^{p'}(T)}$.
	
	For $k = \ell + 1$ in \eqref{def:discretePairing1}, the definition of $\tau_h \in \BDM_{k}(T;\M)$ follows \cite[Proposition 2.3.1]{BoffiBrezziFortin2013} and previous arguments apply verbatim.
	
	\emph{Proof of (c).} For any $v_h = (v_\Tcal,v_\Fcal) \in V_h$, the estimate $\|v_\Tcal\|_{L^p(\Omega)} \lesssim \|v_h\|_h$ from \cite[Proposition 5.4]{DiPietroDroniou2017} and the stability of $R$ in (b) prove $\|v_\Tcal\|_{L^p(\Omega)} \lesssim \|R v_h\|_{L^p(\Omega)}$.
	
	\emph{Proof of (d).} Since $[\tau_h \nu_F]_F \in P_k(\Fcal;\R^m)$ and $\div_\pw \tau_h \in P_\ell(\Tcal;\R^m)$ for any $\tau_h \in \Sigma(\Tcal)$ and $F \in \Fcal$, an integration by parts in \eqref{def:grRecGeneral} implies $(R \I v, \tau_h)_{L^2(\Omega)} = (\D v, \tau_h)_{L^2(\Omega)}$ for all $v \in V$.
\end{proof}
\subsection{Discrete minimization}
Recall $V_h$ from \eqref{def:discretePairing}, $\Sigma(\Tcal)$ from \eqref{def:discretePairing1}, and the discrete energy $E_h$ from \eqref{pr:disMinEnergy}.
The discrete stress approximation is unique and (globally) $W^{p'}(\div,\Omega)$ conforming in the following sense.
%\begin{align}
%	E_h(v_h) \coloneqq \int_{\Omega} W(R v_h) \d{x} - \int_\Omega f\cdot v_\Tcal \d{x} \quad\text{for } v_h = (v_\Tcal,v_\Fcal) \in V_h. \label{def:disMinEnergy}
%\end{align}
\begin{theorem}\label{lem:uniquenessDiscrStress}
	The minimal discrete energy $\min E_h(V_h)$ is attained. Any discrete minimizer $u_h \in V_h$ and the discrete stress $\sigma_h \coloneqq \Pi_{\Sigma(\Tcal)} \D W(R u_h)$ satisfy (a)--(c) with positive constants $\cnstL{cnst:Ruh}$ and $\cnstL{cnst:sigmah}$.
	\begin{enumerate}[wide]
		\item[(a)] Suppose that $u_1 \in V_h$ and $u_2 \in V_h$ minimize \eqref{pr:disMinEnergy}, then $\D W(R u_1) = \D W(R u_2)$ a.e.~in $\Omega$. In particular, the definition of $\sigma_h = \Pi_{\Sigma(\Tcal)} \D W(R u_h)$ does not depend on the choice of the minimizer $u_h$ of \eqref{pr:disMinEnergy}.
		\item[(b)] $\sigma_h \in Q(f,\Tcal) \coloneqq \{\tau_h \in \Sigma(\Tcal) \cap W^{p'}(\div,\Omega;\M):\div \tau_h = -\Pi_\Tcal^\ell f\}$.
		\item[(c)] $\|R u_h\|_{L^p(\Omega)} \leq \newcnstL\label{cnst:Ruh}$ and $\|\sigma_h\|_{L^{p'}(\Omega)} \leq \newcnstL\label{cnst:sigmah}$.
	\end{enumerate}
\end{theorem}
\begin{proof}
	For any $v_h = (v_\Tcal,v_\Fcal) \in V_h$,
	the lower bound of $W$ in \eqref{ineq:growthW} and the discrete Friedrichs' inequality in \Cref{lem:bondednessStabilityGradRec}.c imply
	\begin{align}
		\cnstS{cnst:growthWLeft1}\|R v_h\|_{L^p(\Omega)}^p - \cnstS{cnst:growthWLeft2}|\Omega| - C_{\mathrm{dF}}\|f\|_{L^{p'}(\Omega)}\|R v_h\|_{L^p(\Omega)} \leq E_h(v_h).\label{ineq:lowerBoundDiscreteEnergy}
	\end{align}
	The Young inequality shows for $1 < p < \infty$ that $\inf E_h(V_h) > -\infty$.
	The direct method of the calculus of variations \cite[Theorem 3.30]{Dacorogna2008} proves the existence of discrete minimizers.
	
	\emph{Proof of (a)}. The discrete stress $\sigma_h \coloneqq \Pi_{\Sigma(\Tcal)} \D W(R u_h) \in \Sigma(\Tcal)$ for any discrete minimizer $u_h \in V_h$ satisfies the discrete Euler-Lagrange equations
	\begin{align}
	\int_{\Omega} \sigma_h:R v_h \d{x} = \int_\Omega f\cdot v_\Tcal \d{x} \quad\text{for any } v_h \in V_h. \label{eq:dELE}
	\end{align}
	The choice $A = R u_1$ and $B = R u_2$ in \eqref{ineq:cc} leads to
	\begin{align}
		\begin{split}
			\|\D W(R u_1) &- \D W(R u_2)\|^r_{L^r(\Omega)} \lesssim \big(1 + \|R u_1\|^s_{L^\infty(\Omega)} + \|R u_2\|^s_{L^\infty(\Omega)}\big)\\ &\qquad\times\int_{\Omega} \big(W(R u_1) - W(R u_2) - \D W(R u_2):(R u_1 - R u_2)\big) \d{x}.\label{ineq:proofUniquenessSigmah}
		\end{split}
	\end{align}
	The discrete Euler-Lagrange equations \eqref{eq:dELE} prove that the integral on the right-hand side of \eqref{ineq:proofUniquenessSigmah} is equal to $E_h(u_1) - E_h(u_2) = 0$.
	Thus, $\D W(R u_1) = \D W(R u_2)$ a.e.~in $\Omega$ and $\Pi_{\Sigma(\Tcal)} \D W(R u_1) = \Pi_{\Sigma(\Tcal)} \D W(R u_2)$.
	
	\emph{Proof of (b)}. Given any $v_\Fcal \in P_k(\Fcal(\Omega);\R^m)$, the choice $v_h = (0,v_\Fcal) \in V_h$ in \eqref{eq:dELE} and the definition of the gradient reconstruction $R$ in \eqref{def:grRecGeneral} prove that
	\begin{align*}
	\sum_{F \in \Fcal(\Omega)} \int_{F} [\sigma_h \nu_F]_F \cdot v_F \d{s} = \int_\Omega \sigma_h:R v_h \d{x} = 0.
	\end{align*}
	This $L^2$ orthogonality of $[\sigma_h \nu_F]_F \perp P_k(F;\R^m)$ shows $[\sigma_h \nu_F]_F = 0$ for any $F \in \Fcal(\Omega)$.
	It is well established that the continuity of the normal components of $\sigma_h \in \Sigma(\Tcal)$ leads to $\sigma_h \in H(\div,\Omega;\M)$; the same argument proves $\sigma_h \in W^{p'}(\div,\Omega;\M)$.
	The choice $v_h = (v_\Tcal,0)$ in \eqref{eq:dELE} for any $v_\Tcal \in P_\ell(\Tcal;\R^m)$ leads to $\div \sigma_h = -\Pi_\Tcal^\ell f$.
	
	\emph{Proof of (c).} 
	The choice $v_h = u_h$ in \eqref{ineq:lowerBoundDiscreteEnergy} and $E_h(u_h) \leq E_h(0)$ prove $\|R u_h\|_{L^p(\Omega)} \leq \cnstL{cnst:Ruh}$ for the positive root \cnstL{cnst:Ruh} of the function $\cnstS{cnst:growthWLeft1}x^p - \cnstS{cnst:growthWLeft2}|\Omega| - C_{\mathrm{dF}}\|f\|_{L^{p'}(\Omega)}x - E_h(0)$ in $x > 0$.
	The stability of $\Pi_{\Sigma(\Tcal)}$ in the $L^{p'}$ norm \cite[Lemma 3.2]{DiPietroDroniou2017} leads to $\|\sigma_h\|_{L^{p'}(\Omega)} \leq \|\Pi_{\Sigma(\Tcal)}\|_{\mathcal{L}(L^{p'}(\Omega;\M))} \|\D W(R u_h)\|_{L^{p'}(\Omega)}$. This and the growth of $\D W$ in \Cref{lem:GrowthConj}.a show $\|\sigma_h\|_{L^{p'}(\Omega)}^{p'} \leq \|\Pi_{\Sigma(\Tcal)}\|_{\mathcal{L}(L^{p'}(\Omega;\M))}^{p'}(\cnstS{cnst:growthDW}\cnstL{cnst:Ruh}^p + \cnstS{cnst:growthDWCnst}|\Omega|) \eqqcolon \cnstL{cnst:sigmah}^{p'}$.
\end{proof}
\begin{remark}[global $H(\div)$]
	Any $\sigma_h \in \Sigma(\Tcal)$ that fulfils the discrete Euler-Lagrange equations \eqref{eq:dELE}, satisfies $\sigma_h \in W^{p'}(\div,\Omega;\M)$ with $\div \sigma_h = -\Pi_{\Tcal}^\ell f$. \Cref{lem:uniquenessDiscrStress}.b is not restricted to the minimization problems from \Cref{sec:energyDensity} and, in particular, also applies to the example in \cite{AbbasErnPignet2018}.
\end{remark}
\begin{remark}[mixed FEM]\label{rem:MFEM}
	The mixed finite element scheme of \cite{CGuentherRabus2012} seeks $\sigma_\mathrm{M} = \arg\max E^*(Q(f,\Tcal))$.
	Since $\sigma_\mathrm{M}$ minimizes $\int_\Omega W^*(\bullet) \d{x}$ in $Q(f,\Tcal)$, its subdifferential $\varrho \in \partial W^*(\sigma_\mathrm{M})$ is perpendicular to $Q(0,\Tcal)$ in $L^2(\Omega;\M)$. The definition of $R V_h \subset \Sigma(\Tcal)$ leads to a characterization of any $\tau_\mathrm{M} \in Q(0,\Tcal) = \{\tau_h \in \RT_k(\Tcal;\M): \div \tau_h = 0\}$ by $(\tau_\mathrm{M}, R v_h)_{L^2(\Omega)} = 0$ for all $v_h \in V_h$. This proves $(R V_h)^{\perp} = Q(0,\Tcal)$ in $\Sigma(\Tcal)$. In particular, $\Pi_{\Sigma(\Tcal)} \varrho \in R V_h$. Thus, the mixed FEM seeks $\sigma_\mathrm{M} \in Q(f,\Tcal)$ and $u_\mathrm{M} \in V_h$ with $R u_\mathrm{M} \in \Pi_{\Sigma(\Tcal)} \partial W^*(\sigma_\mathrm{M})$.
\end{remark}
\begin{remark}[comparison to mixed FEM]\label{rem:compareMFEM}
	The unstabilized HHO method of this paper can be rewritten to seek $\sigma_h \in Q(f,\Tcal)$ and $u_h \in V_h$ with $\sigma_h = \Pi_{\Sigma(\Tcal)} \D W(R u_h)$. The two schemes are hence equivalent for a linear problem with a quadratic $W$. But in general, $B \in \Pi_{\Sigma(\Tcal)} \D W(A)$ is \emph{not} equivalent to $A \in \Pi_{\Sigma(\Tcal)} \partial W^*(B)$ for $A, B \in \Sigma(\Tcal)$, because $\Sigma(\Tcal)$ consists of non-constant functions.		
\end{remark}
\begin{remark}[hybridization]\label{rem:hybridization}
	The HHO methodology allows for the elimination of the volume variable as follows: Given $u_\Fcal \in P_k(\Fcal(\Omega);\R^m)$, minimize the convex function $\int_T (W(R u_h) - f \cdot u_\Tcal) \d{x}$ with respect to $u_T \in P_\ell(T;\R^m)$. The solution $u_T = U(u_\Fcal)$ depends on $u_\Fcal$ and leads to $u_h = (u_\Tcal, u_\Fcal) = (U(u_\Fcal), u_\Fcal) \in V_h$. The zero set of the map $G: u_\Fcal \mapsto ([\Pi_{\Sigma(\Tcal)} \D W(R u_h) \nu_F]_F : F \in \Fcal(\Omega))$ characterizes a minimizer $u_h$ of $E_h$ in $V_h$. A numerical implementation may utilize a linearization of $G$ in terms of the skeletal variables in $P_k(\Fcal(\Omega);\R^m)$ as in \cite[Section 3.5]{AbbasErnPignet2018}.
\end{remark}
\section{Error analysis}\label{sec:error} 
Throughout the remaining sections, suppose that $1 + s/p \leq t < r$, $1/t + 1/t' = 1$, and $u \in W^{1,r/(r-t)}_0(\Omega;\R^m)$. This standard assumption on the parameters $r,s,t$ \cite{CPlechac1997,CMueller02,Knees2008} follows a rule of thumb on the growth of $W$ in \eqref{ineq:cc} and holds in all six examples of \cite{Knees2008} and, in particular, in all examples of \Cref{sec:examples}.
\subsection{A~priori error analysis}\label{sec:apriori}
Recall the continuous energy $E$ from \eqref{intro:energy}, the discrete energy $E_h$ from \eqref{pr:disMinEnergy}, and the dual energy $E^*$ from \eqref{def:dualEnergy}. The subsequent a~priori error estimate is analog to \cite[Theorem 2]{CPlechac1997} for conforming FEMs.
\begin{theorem}[a~priori]\label{thm:aprioriGeneral}
	Let $u_h\in V_h$ be a discrete minimizer of $E_h$ in $V_h$. The (unique) discrete stress $\sigma_h = \Pi_{\Sigma(\Tcal)} \D W(R u_h) \in Q(f,\Tcal)$ satisfies (a)--(b) with positive constants $\cnstL{cnst:aprioriLeft2}, \dots, \cnstL{cnst:aprioriRight}$.
	\begin{enumerate}[wide]
		\item[(a)] $\begin{aligned}[t]
			\max\{\cnstL{cnst:aprioriLeft2}^{-1}\|\sigma - \sigma_h\|^r_{L^{r/t}(\Omega)},\cnstL{cnst:aprioriLeft1}^{-1}\|\sigma - \D W(R u_h)\|^r_{L^{r/t}(\Omega)}\}
			\end{aligned} \leq \mathrm{RHS} \coloneqq\\
			\begin{aligned}
			\quad\qquad E^*(\sigma) - \max E^*(Q(f,\Tcal)) + \cnstL{cnst:DuCP} \osc_\ell(f,\Tcal) + \cnstL{cnst:aprioriRight}\|(1 - \Pi_{\Sigma(\Tcal)})\D u\|^{r'}_{L^{r/(r-t)}(\Omega)}.
			\end{aligned}$
		\item[(b)] $\begin{aligned}[t]
			&|E(u) - E_h(u_h)| \leq \max\Big\{E^*(\sigma) - \max E^*(Q(f,\Tcal)),\\
			&\qquad\qquad\cnstL{cnst:DuCP} \osc_\ell(f,\Tcal) + \cnstL{cnst:aprioriRight}\|(1 - \Pi_{\Sigma(\Tcal)})\D u\|^{r'}_{L^{r/(r-t)}(\Omega)} + \frac{r'}{r}\mathrm{RHS}\Big\}.
		\end{aligned}$
	\end{enumerate}
\end{theorem}
Before the remaining parts of this subsection prove \Cref{thm:aprioriGeneral}, it is important to realize that \Cref{thm:aprioriGeneral} is more general than \Cref{thm:apriori}
\begin{proof}[Proof of \Cref{thm:apriori}]
	In the examples from \Cref{sec:examples}, $(p-1)r = p + s$ holds and the choice $t = 1+s/p$ in \Cref{thm:aprioriGeneral} leads to $r/t = p'$ and $r/(r-t) = p$. Since always $u \in W^{1,p}_0(\Omega;\R^m)$, \Cref{thm:apriori} follows from \Cref{thm:aprioriGeneral}.
\end{proof}
The subsequent lemma summarizes two technical tools for the a~priori error analysis.
\begin{lemma}\label{lem:cc}
	Let $1 + s/p \leq t < r$ and $1/t + 1/t' = 1$. For any $\tau, \phi \in L^{p'}(\Omega;\M)$, there exist $\xi,\varrho \in L^p(\Omega;\M)$ such that $\xi \in \partial W^*(\tau)$ a.e.~and $\varrho \in \partial W^*(\phi)$ a.e.~in $\Omega$ with
	\begin{align}
	\begin{split}
	\|\tau - \phi\|^r_{L^{r/t}(\Omega)} &\leq \max\{3,3^{t/t'}\}\cnstS{cnst:cc}\big(|\Omega| + \|\xi\|_{L^p(\Omega)}^p + \|\varrho\|^p_{L^p(\Omega)}\big)^{t/t'}\\
	&\qquad\qquad\times\int_{\Omega} (W^*(\phi) - W^*(\tau) - \xi:(\phi - \tau)) \d{x}.
	\end{split} \label{ineq:ccDualStressAppr}
	\end{align}
	Moreover, any $\xi,\varrho \in L^p(\Omega;\M)$ satisfy
	\begin{align}
		\begin{split}
		\|\D W(\xi) - \D W(\varrho)\|^r_{L^{r/t}(\Omega)} &\leq \max\{3,3^{t/t'}\} \cnstS{cnst:cc} \big(|\Omega| + \|\xi\|^p_{L^p(\Omega)} + \|\varrho\|^p_{L^p(\Omega)}\big)^{t/t'}\\
		&\qquad\times\int_{\Omega} (W(\varrho) - W(\xi) - \D W(\xi):(\varrho - \xi)) \d{x}.
		\end{split} \label{ineq:ccPrimalStressAppr}
	\end{align}
\end{lemma}
\begin{proof}[Proof of \Cref{lem:cc}]
	The convex conjugate $W^*$ of $W$ is continuous in $\M$ and $\partial W^*:\M \to 2^\M$ is an outer semicontinuous set-valued, pointwise non-empty function \cite[Proposition 8.6]{RockafellarWets1998}. Since $\partial W^*$ is close-valued, $\partial W^*$ is measurable \cite[Exercise 14.9]{RockafellarWets1998} and there exists a measurable selection $g$ of $\partial W^*$, i.e., the function $g:\M \to \M$ is Borel measurable and $g(F) \in \partial W^*(F)$ for any $F \in \M$ \cite[Corollary 14.6]{RockafellarWets1998}. In particular, $g(\tau) \in \partial W^*(\tau)$ a.e.~in $\Omega$ and $g(\tau)$ is Lebesgue measurable. The growth of $\partial W^*$ in \Cref{lem:GrowthConj}.c leads to $g(\tau) \in L^{p}(\Omega;\M)$.
	
	The proof of \eqref{ineq:ccDualStressAppr} can follow that of \cite[Theorem 2]{CPlechac1997}.
	If $t = 1$, then $s = 0$ and there is nothing to show. Suppose that $t > 1$. The choice $G = \tau$, $H = \phi$, $A = \xi$, $B = \varrho$ in \eqref{ineq:ccDual} leads in the power $1/t$ to
	\begin{align*}
	\|\tau - \phi\|^{r/t}_{L^{r/t}(\Omega)} \leq \cnstS{cnst:cc}^{1/t}\int_{\Omega} (1 + |\xi|^s + |\varrho|^s)^{1/t}(W^*(\phi) - W^*(\tau) - \xi:(\phi - \tau))^{1/t} \d{x}.
	\end{align*}
	Notice from \eqref{def:SubDiff} that $W^*(\phi) - W^*(\tau) - \xi:(\phi - \tau)$ is non-negative a.e~in $\Omega$.
	A H\"older inequality with the exponents $t$ and $t'$ on the right-hand side shows
	\begin{align}
		\begin{split}
			\|\tau - \phi\|^{r}_{L^{r/t}(\Omega)} &\leq \cnstS{cnst:cc}\|(1 + |\xi|^s + |\varrho|^s)^{1/t}\|_{L^{t'}(\Omega)}^t\\
			&\qquad\times \|W^*(\phi) - W^*(\tau) - \xi:(\phi - \tau)\|_{L^1(\Omega)}.
		\end{split}\label{ineq:proofCC}
	\end{align}
	If $1 \leq t'/t$, then $|\bullet|^{t'/t}$ is convex and Jensen's inequality proves that $(1 + |\xi|^s + |\varrho|^s)^{t'/t} \leq 3^{t'/t-1}(1 + |\xi|^{st'/t} + |\varrho|^{st'/t})$. If $t'/t < 1$, an elementary calculation provides $(1 + |\xi|^s + |\varrho|^s)^{t'/t} \leq 1 + |\xi|^{st'/t} + |\varrho|^{st'/t}$.
	Since $st'/t \leq p$ and $0 \leq st'/(pt) \leq 1$ by assumption, Jensen's inequality for the concave function $|\bullet|^{st'/(pt)}$ shows that
	\begin{align*}
	1 + |\xi|^{st'/t} + |\varrho|^{st'/t} \leq 3^{1-st'/(pt)}(1 + |\xi|^{p} + |\varrho|^{p})^{st'/(pt)} \leq 3(1 + |\xi|^p + |\varrho|^p).
	\end{align*}
	Hence, $\|(1 + |\xi|^s + |\varrho|^s)^{1/t}\|_{L^{t'}(\Omega)}^t \leq \max\{3,3^{t/t'}\}(|\Omega| + \|\xi\|^p_{L^p(\Omega)} + \|\varrho\|^p_{L^p(\Omega)})^{t/t'}$. This and \eqref{ineq:proofCC} conclude the proof of \eqref{ineq:ccDualStressAppr}. The proof of \eqref{ineq:ccPrimalStressAppr} is similar, whence omitted.
\end{proof}
\begin{proof}[Proof of \Cref{thm:aprioriGeneral}.a] 
	\emph{Step 1: Comparison with MFEM.} Let $\sigma_\mathrm{M} \in Q(f,\Tcal)$ be the unique solution of the mixed FEM \cite{CGuentherRabus2012,CLiu2015}, that is $\sigma_\mathrm{M}$ maximizes $E^*$ in $Q(f,\Tcal)$ from \Cref{lem:uniquenessDiscrStress}.b. The two-sided growth of $W^*$ in \Cref{lem:GrowthConj}.b shows
	\begin{align*}
	-\cnstS{cnst:growthWStarRight1}\|\sigma_h\|_{L^{p'}(\Omega)}^{p'} - \cnstS{cnst:growthWLeft2}|\Omega| \leq E^*(\sigma_h) \leq E^*(\sigma_\mathrm{M}) \leq -\cnstS{cnst:growthWStarLeft1}\|\sigma_\mathrm{M}\|_{L^{p'}(\Omega)}^{p'} + \cnstS{cnst:growthWRight2}|\Omega|
	\end{align*}
	and $\|\sigma_\mathrm{M}\|_{L^{p'}(\Omega)}^{p'} \leq \cnstS{cnst:growthWStarLeft1}^{-1}\cnstS{cnst:growthWStarRight1}\cnstL{cnst:sigmah}^{p'} + \cnstS{cnst:growthWStarLeft1}^{-1}(\cnstS{cnst:growthWLeft2} + \cnstS{cnst:growthWRight2})|\Omega| \eqqcolon \newcnstL^{p'}\label{cnst:tauh}$.
	\Cref{lem:cc} allows the selection of
	$\varrho \in L^p(\Omega;\M)$ with $\varrho \in \partial W^*(\sigma_\mathrm{M})$ a.e.~in $\Omega$.
	The growth of $\partial W^*$ in \Cref{lem:GrowthConj}.c provides $\|\varrho\|_{L^p(\Omega)}^p \leq \cnstS{cnst:growthDWStar}\cnstL{cnst:tauh}^{p'} + \cnstS{cnst:growthDWStarCnst}|\Omega|$.
	The choice $\tau = \D W(R u_h)$, $\phi = \sigma_\mathrm{M}$, and $\xi = R u_h$ in \eqref{ineq:ccDualStressAppr} proves that $\newcnstL\label{cnst:proofAprioriLeft2} \coloneqq \max\{3,3^{t/t'}\}\cnstS{cnst:cc}\big((1 + \cnstS{cnst:growthDWStarCnst})|\Omega| + \cnstS{cnst:growthDWStar}\cnstL{cnst:tauh}^{p'} + \cnstL{cnst:Ruh}^p\big)^{t/t'}$ satisfies
	\begin{align}
	\begin{split}
	&\cnstL{cnst:proofAprioriLeft2}^{-1}\|\sigma_\mathrm{M} - \D W(R u_h)\|^r_{L^{r/t}(\Omega)}\\
	&\qquad \leq \int_\Omega \big(W^*(\sigma_\mathrm{M}) - W^*(\D W(R u_h)) - R u_h:(\sigma_\mathrm{M} - \D W(R u_h))\big) \d{x}.
	\end{split}\label{ineq:proofAprioriCCStep2}
	\end{align}
	The definition of the gradient reconstruction $R$ in \eqref{def:grRecGeneral} implies
	\begin{align*}
	\int_{\Omega} R u_h:\sigma_\mathrm{M} \d{x} = \int_\Omega f \cdot u_\Tcal \d{x} + \sum_{F \in \Fcal(\Omega)} \int_{F} u_\Fcal \cdot [\sigma_\mathrm{M} \nu_F]_F \d{s}. 
	\end{align*}
	Since $\sigma_\mathrm{M} \in W^{p'}(\div,\Omega;\M)$, the normal jump $[\sigma_\mathrm{M} \nu_F]_F$ across $F$ vanishes a.e.~on $F \in \Fcal(\Omega)$. Hence, $(R u_h,\sigma_\mathrm{M})_{L^2(\Omega)} = (f,u_\Tcal)_{L^2(\Omega)}$. The discrete Euler-Lagrange equations \eqref{eq:dELE} lead to $(R u_h,\sigma_\mathrm{M} - \D W(R u_h))_{L^2(\Omega)} = 0$.
	This and \eqref{ineq:proofAprioriCCStep2} result in
	\begin{align}
	\cnstL{cnst:proofAprioriLeft2}^{-1}\|\sigma_\mathrm{M} - \D W(R u_h)\|^r_{L^{r/t}(\Omega)} \leq \int_\Omega \big(W^*(\sigma_\mathrm{M}) - W^*(\D W(R u_h)) \big) \d{x}.\label{ineq:proofApriori3}
	\end{align}
	The duality $R u_h \in \partial W^*(\D W(R u_h))$ shows $R u_h:\D W(R u_h) = W^*(\D W(R u_h)) + W(R u_h)$ a.e.~in $\Omega$ \cite[Corollary 12.2.2]{Rockafellar1970}. This and the discrete Euler-Lagrange equations \eqref{eq:dELE} imply
	\begin{align*}
	\int_\Omega f\cdot u_\Tcal \d{x} &= \int_\Omega R u_h:\D W(R u_h) \d{x}  = \int_\Omega W^*(\D W(R u_h)) \d{x} + \int_{\Omega} W(R u_h) \d{x}.
	\end{align*}
	This proves $E^*(\D W(R u_h)) = E_h(u_h)$ and \eqref{ineq:proofApriori3} leads to
	\begin{align}
	\cnstL{cnst:proofAprioriLeft2}^{-1}\|\sigma_\mathrm{M} - \D W(R u_h)\|^r_{L^{r/t}(\Omega)} \leq -E^*(\sigma_\mathrm{M}) + E_h(u_h). \label{ineq:PrAPrioriCC2}
	\end{align}
	
	\emph{Step 2: A temporary error estimate.} The choice $\xi = R u_h$, $\varrho = \D u$ in \eqref{ineq:ccPrimalStressAppr} and the bounds $\|\D u\|_{L^p(\Omega)} \leq \cnstL{cnst:Du}$ from \Cref{thm:reviewResults}.c and $\|R u_h\|_{L^p(\Omega)} \leq \cnstL{cnst:Ruh}$ from \Cref{lem:uniquenessDiscrStress}.c show that $\cnstL{cnst:aprioriLeft1} \coloneqq \max\{3,3^{t/t'}\}\big(|\Omega| + \cnstL{cnst:Du}^p + \cnstL{cnst:Ruh}^p\big)^{t/t'}$ satisfies
	\begin{align}
	\begin{split}
		&\cnstL{cnst:aprioriLeft1}^{-1}\|\sigma - \D W(R u_h)\|^r_{L^{r/t}(\Omega)}\\
		&\qquad\qquad\leq \int_\Omega \big(W(\D u) - W(R u_h) - \D W(R u_h):(\D u - R u_h)\big) \d{x}.
	\end{split}\label{ineq:proofApriori1}
	\end{align}
	The definition of the gradient reconstruction $R$ in \eqref{def:grRecGeneral} and the discrete Euler-Lagrange equations \eqref{eq:dELE} prove the $L^2$ orthogonality $\sigma_\mathrm{M} - \D W(R u_h) \perp R V_h$. This and $R \I u = \Pi_{\Sigma(\Tcal)} \D u$ in \Cref{lem:bondednessStabilityGradRec}.d lead to
	\begin{align}
	\begin{split}
		&-\int_\Omega \D W(R u_h):(\D u - R u_h) \d{x}\\
		&\qquad = \int_\Omega (\sigma_\mathrm{M} - \D W(R u_h)): (1 - \Pi_{\Sigma(\Tcal)})\D u \d{x} - \int_\Omega \sigma_\mathrm{M} : (\D u - R u_h) \d{x}.
	\end{split}\label{ineq:proofApriori2}
	\end{align}
	The definition of $R$ in \eqref{def:grRecGeneral} and an integration by parts result in
	\begin{align*}
	\int_\Omega \sigma_\mathrm{M}:(\D u - R u_h) \d{x} = \int_\Omega f \cdot (\Pi_\Tcal^\ell u - u_\Tcal) \d{x}.
	\end{align*}
	The combination of this with \eqref{ineq:proofApriori1}--\eqref{ineq:proofApriori2} provides
	\begin{align*}
		\begin{split}
			\cnstL{cnst:aprioriLeft1}^{-1}\|\sigma - \D W(R u_h)\|^r_{L^{r/t}(\Omega)}&\leq \int_{\Omega} W(\D u) \d{x} - \int_{\Omega} f\cdot \Pi_{\Tcal}^\ell u \d{x} - E_h(u_h)\\
			&\quad\quad +\int_\Omega (\sigma_\mathrm{M} - \D W(R u_h)):(1 - \Pi_{\Sigma(\Tcal)})\D u \d{x}.
		\end{split}
	\end{align*}
	A piecewise application of the Poincar\'e inequality and $\|\D u\|_{L^p(\Omega)} \leq \cnstL{cnst:Du}$ prove
	\begin{align*}
	\int_\Omega (W(\D u) - f\cdot \Pi_{\Tcal}^\ell u) \d{x} &= E(u) + \int_{\Omega} f\cdot (1 - \Pi^\ell_\Tcal) u \d{x}\leq E(u) + \cnstL{cnst:DuCP} \osc_\ell(f,\Tcal)
	\end{align*}
	with $\cnstL{cnst:DuCP} \coloneqq C_P\cnstL{cnst:Du}$.
	Since there is no duality gap $E(u) = E^*(\sigma)$ on the continuous level, the combination of the two previous formulas verifies
	\begin{align}
	\begin{split}
	&\cnstL{cnst:aprioriLeft1}^{-1}\|\sigma - \D W(R u_h)\|^r_{L^{r/t}(\Omega)} \leq E^*(\sigma) - E_h(u_h)\\
	&\qquad\qquad + \cnstL{cnst:DuCP} \osc_\ell(f,\Tcal) + \int_\Omega (\sigma_\mathrm{M} - \D W(R u_h)):(1 - \Pi_{\Sigma(\Tcal)})\D u \d{x}.
	\end{split}\label{ineq:PrAPrioriCC1}
	\end{align}
	
	\emph{Step 3: The final error estimate.} The sum of \eqref{ineq:PrAPrioriCC1} and \eqref{ineq:PrAPrioriCC2}, the Cauchy, H\"older, and Young inequality prove that $\cnstL{cnst:aprioriRight} \coloneqq \cnstL{cnst:proofAprioriLeft2}^{r'-1}/r'$ satisfies
	\begin{align}
	\begin{split}
	\cnstL{cnst:aprioriLeft1}^{-1}&\|\sigma - \D W(R u_h)\|^r_{L^{r/t}(\Omega)} + (r'\cnstL{cnst:proofAprioriLeft2})^{-1}\|\sigma_\mathrm{M} - \D W(R u_h)\|^r_{L^{r/t}(\Omega)}\\
	&\leq E^*(\sigma) - E^*(\sigma_\mathrm{M}) + \cnstL{cnst:DuCP} \osc_\ell(f,\Tcal) + \cnstL{cnst:aprioriRight} \|(1 - \Pi_{\Sigma(\Tcal)}) \D u\|^{r'}_{L^{r/(r-t)}(\Omega)}.
	\end{split}\label{ineq:proofApriori4}
	\end{align}
	The triangle and Jensen inequality for the convex function $|\bullet|^r$ imply
	\begin{align*}
		\|\sigma - \sigma_h\|^r_{L^{r/t}(\Omega)} \leq 2^{r-1}\big(\|\sigma - \D W(R u_h)\|^r_{L^{r/t}(\Omega)} + \|\sigma_h - \D W(R u_h)\|^r_{L^{r/t}(\Omega)}\big).
	\end{align*}
	The triangle inequality and the stability of the $L^2$ projection $\Pi_{\Sigma(\Tcal)}$ in the $L^{r/t}$ norm \cite[Lemma 3.2]{DiPietroDroniou2017} with the operator norm $\|\Pi_{\Sigma(\Tcal)}\|_{\mathcal{L}(L^{r/t}(\Omega;\M))}$ show that
	\begin{align*}
		 \|\sigma_h - \D W(R u_h)\|_{L^{r/t}(\Omega)} &\leq  \|\sigma_h - \sigma_\mathrm{M}\|_{L^{r/t}(\Omega)} +  \|\sigma_\mathrm{M} - \D W(R u_h)\|_{L^{r/t}(\Omega)}\\
%		 &\leq \|\Pi_{\Sigma(\Tcal)}(\sigma_\mathrm{M} - \D W(R u_h))\|_{L^{r/t}(\Omega)} + \|\sigma_\mathrm{M} - \D W(R u_h)\|_{L^{r/t}(\Omega)}\\
		 &\leq \big(1 + \|\Pi_{\Sigma(\Tcal)}\|_{\mathcal{L}(L^{r/t}(\Omega;\M))}\big)\|\sigma_\mathrm{M} - \D W(R u_h)\|_{L^{r/t}(\Omega)}.
	\end{align*}
	The combination of this with \eqref{ineq:proofApriori4} concludes the proof of (a) with the constant $\cnstL{cnst:aprioriLeft2} \coloneqq 2^{r-1}\max\big\{\cnstL{cnst:aprioriLeft1},r'\cnstL{cnst:proofAprioriLeft2}(1 + \|\Pi_{\Sigma(\Tcal)}\|_{\mathcal{L}(L^{r/t}(\Omega;\M))})^r\big\}$.
	
	\emph{Proof of \Cref{thm:aprioriGeneral}.b.} 
	Recall the maximizer $\sigma_\mathrm{M} \in Q(f,\Tcal)$ of $E^*$ in $Q(f,\Tcal)$ from the proof of (a) and $0 \leq -E^*(\sigma_\mathrm{M}) + E_h(u_h)$ from \eqref{ineq:PrAPrioriCC2}. Thus,
	\begin{align}
		E^*(\sigma) - E_h(u_h) \leq E^*(\sigma) - E^*(\sigma_\mathrm{M}).\label{ineq:proofAprioriEnergy1}
	\end{align}
	A weighted Young inequality in \eqref{ineq:PrAPrioriCC1} leads to
	\begin{align*}
	E_h(u_h) - E^*(\sigma) \leq \cnstL{cnst:DuCP} \osc_\ell(f,\Tcal) &+ \frac{1}{r\cnstL{cnst:proofAprioriLeft2}}\|\sigma_\mathrm{M} - \D W(R u_h)\|_{L^{r/t}(\Omega)}^{r}\\
	&+ \cnstL{cnst:aprioriRight} \|(1 - \Pi_{\Sigma(\Tcal)}) \D u\|^{r'}_{L^{r/(r-t)}(\Omega)}.
	\end{align*}
	The combination of this with \eqref{ineq:proofApriori4}--\eqref{ineq:proofAprioriEnergy1} concludes the proof of (b).
\end{proof}
%\begin{remark}[Fortin interpolation of $\sigma$]
%	If $\div \sigma = -f\in L^2(\Omega;\R^m)$ and $\sigma \in L^{q^*}(\Omega;\M)$ for some $q^* > 2$, then the Fortin interpolation $\I_F \sigma$ \cite[Section 2.5.1]{BoffiBrezziFortin2013} exists and \Cref{thm:aprioriGeneral} holds for $\tau_h$ replaced by $\I_F \sigma$. \Cref{thm:reviewResults}.e guarantees $\sigma \in L^{q*}(\Omega;\M)$ for the optimal design problem and the p-Laplace with $1 < p < 4$ and $f \in L^2(\Omega;\R^m)$ in two space dimensions. For the two-well problem or the p-Laplace with $p \geq 4$, the existence of a Fortin interpolation $I_F \sigma$ is \emph{not} guaranteed a~priori in the absence of additional conditions.
%\end{remark}
For smooth functions $\sigma$, $u$, and $f$, the subsequent corollary implies the rate $\|\sigma - \sigma_h\|_{L^{p'}(\Omega)} \lesssim h_{\max}^{(k+1)/r}$ with maximal mesh-size $h_{\max}$ of $\Tcal$.
\begin{corollar}\label{cor:convergenceRate}
	Consider the examples of \Cref{sec:examples} and adapt the notation from \Cref{thm:apriori}.
	If $f \in W^{\ell+1,p'}(\Tcal;\R^m)$, $u \in W^{k+2,p}(\Omega;\R^m)$ for some minimizer $u$ of $E$ in $V$, and $\sigma \coloneqq \D W(\D u) \in W^{k+1,p'}(\Omega;\M)$, then
	\begin{align*}
		&\|\sigma-\sigma_h\|_{L^{p'}(\Omega)}^r + \|\sigma - \D W(R u_h)\|^r_{L^{p'}(\Omega)} + |E(u) - E_h(u_h)|\\
		&~~\lesssim h^{2(\ell+1)}_{\max}\|f\|_{W^{\ell+1,p'}(\Omega)}\|u\|_{W^{\ell+1,p}(\Omega)} + h^{k+1}_{\max} \|\sigma\|_{W^{k+1,p'}(\Omega)} + h^{(k+1) r'}_{\max}\|u\|_{W^{k+2,p}(\Omega)}^{r'}.
	\end{align*}
\end{corollar}
\begin{proof}
	Recall Step 2 in the proof of \Cref{thm:aprioriGeneral} and notice that, for smooth functions $f \in W^{\ell+1,p}(\Omega;\R^m)$, the orthogonality and the best approximation of the piecewise $L^2$ projection imply
	\begin{align}
		\int_{\Omega} f \cdot (1 - \Pi_{\Tcal}^\ell) u \d{x} \lesssim h^{2(\ell+1)}_{\max}\|f\|_{W^{\ell+1,p'}(\Omega)}\|u\|_{W^{\ell+1,p}(\Omega)}.\label{ineq:higherOscillation}
	\end{align}
	This replaces the data oscillation in the a priori estimate from \Cref{thm:aposterioriGeneral}. The approximation property of $L^2$ projections onto piecewise polynomials \cite[Lemma 3.4]{DiPietroDroniou2017} leads to $\|(1 - \Pi_{\Sigma(\Tcal)})\D u\|_{L^p(\Omega)} \lesssim h^{k+1}_{\max} \|u\|_{W^{k+2,p}(\Omega)}$.
	Let $\tau_h \in \arg\min_{\varphi \in Q(f,\Tcal)}\|\sigma - \varphi\|_{L^{p'}(\Omega)}$ and let $\xi \in L^p(\Omega;\M)$ be a measurable selection of $\partial W^*(\tau_h)$ from \Cref{lem:cc}, $\xi \in \partial W^*(\tau_h)$ a.e.~in $\Omega$.
	Since $\|\sigma - \tau_h\|_{L^{p'}(\Omega)} \leq \|\sigma - \sigma_h\|_{L^{p'}(\Omega)} \leq \|\sigma\|_{L^{p'}(\Omega)} + \|\sigma_h\|_{L^{p'}(\Omega)}$, the reverse triangle inequality shows that
	\begin{align*}
		\|\tau_h\|_{L^{p'}(\Omega)} \leq 2\|\sigma\|_{L^{p'}(\Omega)} + \|\sigma_h\|_{L^{p'}(\Omega)} \leq 2\cnstL{cnst:sigma} + \cnstL{cnst:sigmah}
	\end{align*}
	with the constants $\cnstL{cnst:sigma}$ from \Cref{thm:reviewResults}.c and $\cnstL{cnst:sigmah}$ from \Cref{lem:uniquenessDiscrStress}.c.
	The growth of $\partial W^*$ in \Cref{lem:GrowthConj}.c proves $\|\xi\|_{L^p(\Omega)}^p \leq \cnstS{cnst:growthDWStar}(2\cnstL{cnst:sigma} + \cnstL{cnst:sigmah})^{p'} + \cnstS{cnst:growthDWStarCnst}|\Omega| \eqqcolon \newcnstL\label{cnst:xi}$.
	The definition of the subdifferential $\partial W^*$ in \eqref{def:SubDiff} leads to
	\begin{align}
		\begin{split}
			E^*(\sigma) - E^*(\tau_h) &= \int_{\Omega} (W^*(\tau_h) - W^*(\sigma)) \d{x}\\
			&\leq -\int_{\Omega} \xi:(\sigma - \tau_h) \d{x} \leq \cnstL{cnst:xi}\|\sigma - \tau_h\|_{L^{p'}(\Omega)}.
		\end{split}
		\label{ineq:proof-convergence-rates}
	\end{align}
	The approximation property of the Fortin interpolation $I_F \sigma \in Q(f,\Tcal)$ \cite[Proposition 2.5.4]{BoffiBrezziFortin2013} to $\sigma \in H^{k+1}(\Omega;\M)$ is well-established for $p = 2$ \cite[Proposition 2.5.4]{BoffiBrezziFortin2013}. The same arguments
	lead to $\|\sigma - \tau_h\|_{L^{p'}(\Omega)} \leq \|\sigma - I_F \sigma\|_{L^{p'}(\Omega)} \lesssim h^{k+1}_{\max}\|\sigma\|_{W^{k+1,p'}(\Omega)}$ for $\sigma \in W^{k+1,p'}(\Omega;\M)$. The a priori estimate for $|E(u) - E_h(u_h)|$ in \Cref{thm:aprioriGeneral}.b and previous arguments conclude the proof.
\end{proof}
\begin{remark}[$p$-Laplace]\label{rem:pLaplace}
	Additional control over the primal variable in the $p$-Laplace problem of \Cref{ex:pLaplace} improves the results in \Cref{cor:convergenceRate} as outlined below. For the sake of simplicity, let $k = \ell$ in \eqref{def:discretePairing1}.
	If $1 < p < 2$, the bound $|a - b|^2 \lesssim (|a|^{2-p} + |b|^{2-p})(\D W(a) - \D W(b)) \cdot (a - b)$ in \cite[Lemma 5.2]{GlowinskiMarrocco1975} for any $a,b \in \R^n$ and the arguments in the proof of \Cref{lem:cc} verify for all $\xi, \varrho \in L^p(\Omega)$ that
	\begin{align}
	\|\xi - \varrho\|^2_{L^p(\Omega)} \lesssim (\|\xi\|^p_{L^p(\Omega)} + \|\varrho\|_{L^p(\Omega)}^p)^{\frac{2-p}{p}}\int_\Omega (\D W(\xi) - \D W(\varrho))\cdot(\xi - \varrho) \d{x}.
	\label{ineq:cc-p-Laplace-case-1}
	\end{align}
	Recall $\tau_h \in Q(f,\Tcal)$ and $\xi = \D W^*(\tau_h)$ from the proof of \Cref{cor:convergenceRate}. The choice $\varrho \coloneqq \D u$ in \eqref{ineq:cc-p-Laplace-case-1} proves $\|\D u - \xi\|_{L^p(\Omega)} \lesssim \|\sigma - \tau_h\|_{L^{p'}(\Omega)}$. This, an integration by parts, and a piecewise application of the Poincar\'e inequality in \eqref{ineq:proof-convergence-rates} lead to
	\begin{align*}
	E^*(\sigma) - E^*(\tau_h) &\leq \int_{\Omega} (\D u - \xi) \cdot (\sigma - \tau_h) \d{x} - \int_\Omega \D u \cdot (\sigma - \tau_h) \d{x}\\
	&\lesssim \|\sigma - \tau_h\|_{L^{p'}(\Omega)}^2 + \osc_k(f,\Tcal).
	\end{align*}
	This and the rates in the proof of \Cref{cor:convergenceRate} confirm $\|\sigma - \sigma_h\|_{L^{p'}(\Omega)} + \|\sigma - \D W(R u_h)\|_{L^{p'}(\Omega)} \lesssim h_{\max}^{(k+1)(p-1)}$. Moreover, the convexity control \eqref{ineq:cc-p-Laplace-case-1} and \Cref{rem:monotonicity} prove 
	\begin{align}
	\begin{split}
	\|\D u - R u_h\|_{L^p(\Omega)}^2 &\lesssim \int_\Omega (W(\D u) - W(R u_h) - \D W(R u_h) \cdot (\D u - R u_h)) \d{x},\\
	\|R u_h - \varrho\|_{L^p(\Omega)}^2 &\lesssim \int_\Omega (W^*(\sigma_\mathrm{M}) - W^*(\D W(R u_h)) - R u_h \cdot (\sigma_\mathrm{M} - \sigma_h)) \d{x}
	\end{split}
	\label{ineq:proof-convergence-rates-cc-primal}
	\end{align}
	with $\sigma_\mathrm{M} = \arg\max E^*(Q(f,\Tcal))$ and $\varrho = \D W^*(\sigma_\mathrm{M})$ from Step 1 of the proof of \Cref{thm:aprioriGeneral}. The arguments from the proof of \Cref{thm:aprioriGeneral}.a apply to the right-hand sides of \eqref{ineq:proof-convergence-rates-cc-primal} and imply 
	\begin{align}
	\begin{split}
	&\|\D u - R u_h\|_{L^p(\Omega)}^2 + \|R u_h - \varrho\|_{L^p(\Omega)}^2 \lesssim E^*(\sigma) - E^*(\sigma_\mathrm{M})\\
	&\qquad\qquad + \cnstL{cnst:DuCP}\osc_k(f,\Tcal) + \|\sigma_\mathrm{M} - \D W(R u_h)\|_{L^{p'}(\Omega)} \|(1 - \Pi_{\Sigma(\Tcal)}) \D u\|_{L^p(\Omega)}.
	\end{split}
	\label{ineq:proof-a-priori-p-Laplace}
	\end{align}
	The convexity control \eqref{ineq:ccMonotonicty} shows $\|\sigma_\mathrm{M} - \D W(R u_h)\|_{L^p(\Omega)} \lesssim \|R u_h - \varrho\|_{L^p(\Omega)}^{p-1}$. This, a Young inequality with exponents $2/(p-1)$ and $2/(3-p)$ on the right-hand side of \eqref{ineq:proof-a-priori-p-Laplace}, and the arguments from the proof of \Cref{cor:convergenceRate} verify $\|\D u - R u_h\|_{L^p(\Omega)} \lesssim h_{\max}^{(k+1)/(3-p)}$. This improves the existing rate $h_{\max}^{(k+1)(p-1)}$ in \cite{DiPietroDroniou2017-II, DiPietroDroniouManzini2018}. For $2 \leq p < \infty$, the arguments of this paper lead to
	$\|\sigma - \sigma_h\|_{L^{p'}(\Omega)} + \|\sigma - \D W(R u_h)\|_{L^{p'}(\Omega)} \lesssim h_{\max}^{(k+1)p'/2}$ and $\|\D u - R u_h\|_{L^p(\Omega)} \lesssim h_{\max}^{(k+1)/(p-1)}$. This confirm the results in \cite{DiPietroDroniou2017-II, DiPietroDroniouManzini2018}.
\end{remark}
\begin{remark}[reduced convergence rates]
	Notice that the optimal stress approximation of $\Sigma(\Tcal)$ in $L^{p'}(\Omega;\M)$ is of order $k+1$, but \Cref{cor:convergenceRate} solely guarantees a convergence rate of order $(k+1)/r$. This reduction is also observed in \cite[Theorem 5.2]{CGuentherRabus2012} for a lowest-order Raviart-Thomas discretization of the optimal design problem.
\end{remark}
%\begin{remark}
%	Suppose that $\varrho \in L^p(\Omega;\R^{m \times n})$ with $\varrho \in \partial W^*(I_F \sigma)$ a.e.~in $\Omega$, then
%	\begin{align*}
%	E^*(\sigma) - E^*(I_F \sigma) \leq -\int_\Omega \xi:(\sigma - I_F \sigma) \d{x} \leq \|\xi\|_{L^{p}(\Omega)} \|(1 - I_F) \sigma\|_{L^{p'}(\Omega)}.
%	\end{align*}
%	If $\sigma$ is sufficiently smooth, e.g.~$\sigma \in W^{\ell+1,p'}(\Omega;\M)$, then $\|(1 - I_F) \sigma\|_{L^{p'}(\Omega)} \lesssim h^{\ell+1}\|\sigma\|_{W^{\ell+1,p'}(\Omega)}$. Note that $\|\sigma - \sigma_h\|^r_{L^{r/t}(\Omega)}$ converges with order $h^{r(\ell+1)}$ for elliptic problems, however, this cannot be expected here due to the degeneracy of $W$. Some explicit computations in \cite{CGuentherRabus2012} shows that the convergence rate depends on the (unknown) size of the microstructure.
%\end{remark}
\subsection{A~posteriori error analysis}\label{sec:aposteriori}
Let $u$ minimize $E$ in $V$ and let $v \in V$ be arbitrary. The choice $\xi = \D u$, $\varrho = \D v$ in \eqref{ineq:ccPrimalStressAppr}, and the Euler-Lagrange equations \eqref{eq:ELE} lead to the estimate
\begin{align*}
\|\sigma - \D W(\D v)\|^r_{L^{p'}(\Omega)} \leq \max\{3,3^{t/t'}\}\cnstS{cnst:cc}\big(|\Omega| + \cnstL{cnst:Du}^p + \|\D v\|^p_{L^p(\Omega)}\big)^{t/t'}\big(E(v) - E(u)\big).
\end{align*}
Provided $E(u) = \min E(V)$ has a known lower energy bound, this provides an a~posteriori stress error estimate in a conforming discretization for the approximation $v \in V$ (even for inexact solve) and its (computable) energy $E(v)$. This technique is employed e.g.~in \cite[Section 10.2.5]{Bartels2015}. Nonconforming, mixed, and HHO discretizations can be utilized for lower energy bounds (LEBs).
\begin{theorem}[a~posteriori]\label{thm:aposterioriGeneral}
	Let $u_h$ minimize $E_h$ in $V_h$. The unique discrete stress $\sigma_h \coloneqq \Pi_{\Sigma(\Tcal)} \D W(R u_h)$ and any $v \in W^{1,r/(r-t)}_0(\Omega;\R^m)$ satisfy
	\begin{enumerate}[wide]
		\item[(a)] (LEB) $\cnstL{cnst:LEB}^{-1}\|\sigma - \sigma_h\|^r_{L^{r/t}(\Omega)} + E^*(\sigma_h) - \cnstL{cnst:DuCP} \osc_\ell(f,\Tcal) \leq \min E(V)$;
		\item[(b)] $\begin{aligned}[t]
		&\cnstL{cnst:aposterioriLeft1}^{-1}\|\sigma - \sigma_h\|^r_{L^{r/t}(\Omega)} + \cnstL{cnst:aprioriLeft1}^{-1}\|\sigma - \D W(R u_h)\|^r_{L^{r/t}(\Omega)} \leq E_{h}(u_h) - E^*(\sigma_h)\\
		&\qquad\qquad + \cnstL{cnst:DuCP} \textup{osc}_\ell(f,\mathcal{T}) + \cnstL{cnst:aposterioriRight} \|R u_h - \D v\|_{L^{r/(r-t)}(\Omega)}^{r'} - \int_\Omega f\cdot(1-\Pi^\ell_\Tcal)v \d{x};
		\end{aligned}$
		\item[(c)] $\begin{aligned}[t]
			|E(u) - E_h(u_h)| \leq \max\Big\{\cnstL{cnst:Du}\|\sigma_h - \D W(R u_h)\|_{L^{p'}(\Omega)} + \cnstL{cnst:DuCP}\osc_\ell(f,\Tcal),
		\end{aligned}\\
		\begin{aligned}
			\quad\cnstL{cnst:aposterioriLeft1}^{-1}\|\sigma-\sigma_h\|^{r}_{L^{r/t}(\Omega)} + (r'/r)^{r'-1}\cnstL{cnst:aposterioriRight}\|R u_h - \D v\|^{r'}_{L^{r/(r-t)}(\Omega)} -\int_{\Omega} f\cdot(1-\Pi_{\Tcal}^\ell) v \d{x}\Big\}.
		\end{aligned}$
	\end{enumerate}
\end{theorem}
Before the remaining parts of this subsection focus on the proof of \Cref{thm:aposterioriGeneral}, notice that \Cref{thm:aposterioriGeneral} implies \Cref{thm:aposteriori}.
\begin{proof}[Proof of \Cref{thm:aposteriori}]
	The choice $t = 1+s/p$ in \Cref{thm:aposterioriGeneral} for the examples from \Cref{sec:examples} leads to $r/t = p'$ and $r/(r-t) = p$ and proves \Cref{thm:aposteriori}.
\end{proof}
\begin{proof}[Proof of \Cref{thm:aposterioriGeneral}.a]
	\emph{(LEB)} Let $\varrho \in L^p(\Omega;\M)$ be a measurable selection of $\partial W^*(\sigma_{h})$ with $\varrho \in \partial W^*(\sigma_h)$ a.e.~in $\Omega$ from \Cref{lem:cc}. The growth of $\partial W^*$ in \Cref{lem:GrowthConj}.c and $\|\sigma_h\|_{L^{p'}(\Omega)} \leq \cnstL{cnst:sigmah}$ from \Cref{lem:uniquenessDiscrStress}.c lead to $\|\varrho\|_{L^p(\Omega)}^p \leq \cnstS{cnst:growthDWStar}\cnstL{cnst:sigmah}^{p'} + \cnstS{cnst:growthDWStarCnst}|\Omega|$.
	The choice $\tau = \sigma$, $\phi = \sigma_h$, and $\xi = \D u$ in \eqref{ineq:ccDualStressAppr} proves
	\begin{align}
		\cnstL{cnst:LEB}^{-1}\|\sigma - \sigma_h\|^r_{L^{r/t}(\Omega)} &\leq \int_{\Omega} (W^*(\sigma_h) - W^*(\sigma) - \D u : (\sigma_h - \sigma)) \d{x}\label{ineq:proofLEB}
	\end{align}
	for $\cnstL{cnst:LEB} \coloneqq \max\{3,3^{t/t'}\}\cnstS{cnst:cc}\big((1 + \cnstS{cnst:growthDWStarCnst})|\Omega| + \cnstL{cnst:Du}^p + \cnstS{cnst:growthDWStar}\cnstL{cnst:sigmah}^{p'}\big)^{t/t'}$.
	An integration by parts plus a piecewise application of the Poincar\'e inequality with $\cnstL{cnst:DuCP} = C_P \cnstL{cnst:Du}$ from Step 1 of the proof of \Cref{thm:aprioriGeneral} show that
	\begin{align*}
	-\int_{\Omega} \D u : (\sigma_h - \sigma) \d{x} = \int_{\Omega} u \cdot (1 - \Pi_\Tcal^\ell) f \d{x} \leq \cnstL{cnst:DuCP} \osc_\ell(f,\Tcal).
	\end{align*}
	This, \eqref{ineq:proofLEB}, and $E^*(\sigma) = E(u)$ imply the lower energy bound
	\begin{align}
		\cnstL{cnst:LEB}^{-1}\|\sigma - \sigma_h\|^r_{L^{r/t}(\Omega)} \leq E(u) -  E^*(\sigma_h) + \cnstL{cnst:DuCP} \osc_\ell(f,\Tcal). \label{ineq:aposterioriUnstabProofCC1}
	\end{align}
	
	\emph{Proof of \Cref{thm:aposterioriGeneral}.b.} The choice $\tau = \D W(R u_h)$, $\phi = \sigma$, $\xi = R u_h$, and $\varrho = \D u$ in \eqref{ineq:ccDualStressAppr}, and the $L^2$ orthogonality $\sigma_h - \D W(R u_h) \perp \Sigma(\Tcal)$ show that
	\begin{align}
	\cnstL{cnst:aprioriLeft1}^{-1}&\|\sigma - \D W(Ru_h)\|^r_{L^{r/t}(\Omega)}\nonumber\\
	&\leq \int_{\Omega} (W^*(\sigma) - W^*(\D W(R u_h)) - R u_h:(\sigma - \D W(R u_h))) \d{x}\nonumber\\
	&= E^*(\D W(R u_h)) - E(u) - \int_{\Omega} \big((R u_h - \D v):(\sigma - \sigma_h) + \D v:(\sigma - \sigma_{h})\big) \d{x} \label{ineq:aposterioriUnstabProofCC2}
	\end{align}
	for all $v \in W^{1,r/(r-t)}_0(\Omega;\R^m)$.
	Recall $E^*(\D W(R u_h)) = E_h(u_h)$ from the proof of \Cref{thm:aprioriGeneral}. An integration by parts proves
	\begin{align*}
	-\int_{\Omega} \D v:(\sigma - \sigma_h) \d{x} = -\int_{\Omega} f\cdot(1-\Pi_\Tcal^\ell)v \d{x}.
	\end{align*}
	This, the sum of \eqref{ineq:aposterioriUnstabProofCC1}--\eqref{ineq:aposterioriUnstabProofCC2}, and a weighted Young inequality imply
	\begin{align*}
	&\cnstL{cnst:LEB}^{-1}\|\sigma - \sigma_h\|^r_{L^{r/t}(\Omega)} + \cnstL{cnst:aprioriLeft1}^{-1}\|\sigma - \D W(R u_h)\|^r_{L^{r/t}(\Omega)}\\
	&\qquad\leq E_{h}(u_h) - E^*(\sigma_h) + \cnstL{cnst:DuCP} \textup{osc}_\ell(f,\mathcal{T}) + \frac{\varepsilon}{r}\|\sigma - \sigma_h\|_{L^{r/t}(\Omega)}^r\\
	&\qquad\qquad + \frac{\varepsilon^{1-r'}}{r'} \|R u_h - \D v\|_{L^{r/(r-t)}(\Omega)}^{r'} - \int_\Omega f\cdot(1-\Pi^k_\Tcal)v \d{x}
	\end{align*}
	for $\varepsilon \coloneqq \cnstL{cnst:LEB}^{-1}$. This proves (b) with $\cnstL{cnst:aposterioriLeft1}\coloneqq r'\cnstL{cnst:LEB}$ and $\cnstL{cnst:aposterioriRight} \coloneqq \cnstL{cnst:LEB}^{r'-1}/r'$.
	
	\emph{Proof of \Cref{thm:aposterioriGeneral}.c} The definition of the convex conjugate $W^*$ shows
	\begin{align}
		-\int_{\Omega} \big(W^*(\D W(R u_h)) + W(\D u)\big) \d{x} \leq - \int_{\Omega} \D W(R u_h):\D u \d{x}.\label{ineq:proofAposterioriEnergyStep1.1}
	\end{align}
	An integration by parts, the Cauchy inequality, and a piecewise application of the Poincar\'e inequality as in the proof of (a) lead to
	\begin{align}
		- \int_{\Omega} \D W(R u_h):\D u \d{x} &= \int_{\Omega} (\sigma_h - \D W(R u_h)):\D u \d{x} + \int_{\Omega} \Pi_{\Tcal}^\ell f \cdot u \d{x}\nonumber\\
		&\leq \cnstL{cnst:Du}\|\sigma_h - \D W(R u_h)\|_{L^{p'}(\Omega)} + \int_{\Omega} f \cdot u \d{x} + \cnstL{cnst:DuCP}\osc_\ell(f,\Tcal).\label{ineq:proofAposterioriEnergyStep1.2}
	\end{align}
	Recall $E^*(\D W(R u_h)) = E_h(u_h)$ from Step 1 of the proof of \Cref{thm:aprioriGeneral}. The combination of this and \eqref{ineq:proofAposterioriEnergyStep1.1}--\eqref{ineq:proofAposterioriEnergyStep1.2} results in
	\begin{align}
		E_h(u_h) - E(u) \leq \cnstL{cnst:Du}\|\sigma_h - \D W(R u_h)\|_{L^{p'}(\Omega)} + \cnstL{cnst:DuCP}\osc_\ell(f,\Tcal).\label{ineq:proofAposterioriEnergyStep1}
	\end{align}
	An integration by parts in \eqref{ineq:aposterioriUnstabProofCC2} and a weighted Young inequality show
	\begin{align*}
		E(u) - E_h(u_h) &\leq \cnstL{cnst:aposterioriLeft1}^{-1}\|\sigma-\sigma_h\|^{r}_{L^{r/t}(\Omega)}\nonumber\\
		&\quad + (r'/r)^{r'-1}\cnstL{cnst:aposterioriRight}\|R u_h - \D v\|^{r'}_{L^{r/(r-t)}(\Omega)} -\int_{\Omega} f\cdot(1-\Pi_{\Tcal}^\ell) v \d{x}
	\end{align*}
	for all $v \in W^{1,r/(r-t)}_0(\Omega;\R^m)$.
	This and \eqref{ineq:proofAposterioriEnergyStep1} conclude the proof of (c).
\end{proof}
\begin{remark}[superlinear convergent LEB]\label{rem:LEB}
	The arguments from the proof of \Cref{cor:convergenceRate} verify $E^*(\sigma) - E^*(\sigma_h) \lesssim \|\sigma - \sigma_h\|_{L^{r/t}(\Omega)}$. In particular, the lower energy bound can converge superlinearly towards $E(u)$, which is observed in all numerical benchmarks of \Cref{sec:numericalExamples}.
\end{remark}
\begin{remark}[discrete duality gap]\label{rem:discreteDualityGap}
	The discrete lowest-order mixed FEM for the optimal design problem in \cite{CLiu2015} has no discrete duality gap to a nonconforming Crouzeix-Raviart FEM for the primal minimization problem \cite[Theorem 3.1]{CLiu2015}. This is restricted to the lowest-order case and cannot be expected here. In fact, recall the maximizer $\sigma_\mathrm{M}$ of $E^*$ in $Q(f,\Tcal)$ and let $\varrho \in L^p(\Omega;\M)$ with $\varrho \in \partial W^*(\sigma_\mathrm{M})$ a.e.~in $\Omega$. The choice $\tau = \D W(R u_h)$, $\phi = \sigma_\mathrm{M}$, and $\xi = R u_h$ in \eqref{ineq:ccDualStressAppr} and the definition of $R$ in \eqref{def:grRecGeneral} show that
	\begin{align*}
	&\|\sigma_\mathrm{M} - \D W(R u_h)\|^r_{L^{r/t}(\Omega)}\\
	&\qquad\lesssim \int_\Omega (W^*(\sigma_\mathrm{M}) - W^*(\D W(R u_h))) \d{x} = E_h(u_h) - E^*(\sigma_\mathrm{M}).
	\end{align*}
	It has to be expected for a general nonlinear function $\D W$ that $\D W(R u_h) \notin \Sigma(\Tcal)$. Then $E_h(u_h) - E^*(\sigma_h) \geq E_h(u_h) - E^*(\sigma_\mathrm{M}) > 0$. There is a discrete duality gap.
\end{remark}
\section{Numerical examples}\label{sec:numericalExamples}
Throughout this section, let $V_h = P_k(\Tcal) \times P_k(\Fcal(\Omega))$, $\Sigma(\Tcal) = \RT_k^\pw(\Tcal;\R^n)$, $t = 1 + s/p$, $r/t = p'$, and $r/(r-t) = p$ with $r = r' = 2$ in all examples of \Cref{sec:examples} in 2D.
\subsection{Numerical realization}\label{sec:numerical_realization}
Some remarks on the implementation, the adaptive mesh-refinements, and the output precede the three numerical examples.
\subsubsection{Implementation} The discrete Euler-Lagrange equations \eqref{eq:dELE} have been realized with an iterative solver \texttt{fminunc} from the MATLAB standard library in an extension of the data structures and the short MATLAB programs in \cite{AlbertyCFunken1999,CBrenner2017,AFEM}. The first and (piecewise) second derivatives of $W$ have been provided for the trust-region quasi-Newton scheme with parameters of \texttt{fminunc} set to $\texttt{FunctionTolerance} = \texttt{OptimalityTolerance} = \texttt{StepTolerance} = 10^{-14}$ and $\texttt{MaxIterations} = \texttt{Inf}$ for improved accuracy.

The class of minimization problems at hand allows, in general, for multiple exact and discrete solutions. The numerical experiments select one (of those) by the approximation in \texttt{fminunc} with the initial value computed as follows. On the coarse initial triangulations $\Tcal_0$ from \Cref{fig:initTriangulation}, the initial value $v_h = (v_\Tcal,v_\Fcal) \in V_h$ is defined by $v_\Tcal \equiv 1$ and $v_\Fcal|_F \equiv 1$ on any $F \in \Fcal(\Omega)$. On each refinement $\widehat{\Tcal}$ of some triangulation $\Tcal$, the initial approximation is defined by a prolongation of the output of the call \texttt{fminunc} on the coarse triangulation $\Tcal$. The prolongation maps $(v_\Tcal,v_\Fcal)$ onto $(v_{\widehat{\Tcal}},v_{\widehat{\Fcal}})$ by piecewise $L^2$ projections (from one triangle to a subtriangle or one edge to some subedge) and defines the remaining values $v_{\widehat{F}} \coloneqq \Pi_{\widehat{F}}^k v_T$ for any edge $\widehat{F} \in \widehat{\Fcal}$ across the triangle $T$ in case $\widehat{F} \subset T$ but $\widehat{F} \not\subset \partial T$.

The numerical integration of polynomials is exact with the quadrature formula in \cite{HammerStroud1956}: For non-polynomial functions such as $W(R v_h)$ with $v_h \in V_h$, the number of chosen quadrature points allows for exact integration of polynomials of order $p(k+1)$ with the growth $p$ of $W$ and the polynomial order $k$ of the discretization; the same quadrature formula also applies to the integration of the dual energy density $W^*$. The implementation is based on the in-house AFEM software package in MATLAB \cite{AFEM}.
\begin{figure}[h!]
	\begin{minipage}[t]{0.485\textwidth}
		\centering
		\includegraphics[scale=0.73]{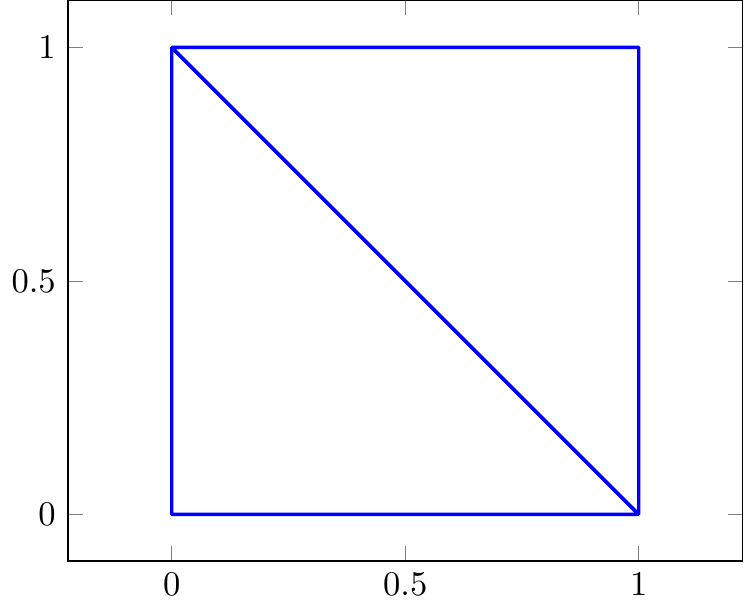}
	\end{minipage}\hfill
	\begin{minipage}[t]{0.485\textwidth}
		\centering
		\includegraphics[scale=0.73]{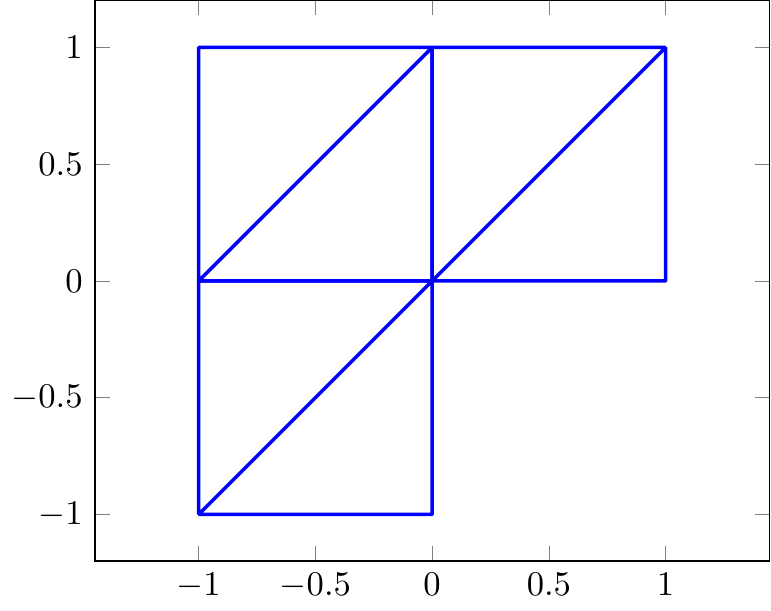}
	\end{minipage}\hfill
	\caption{Initial triangulation $\Tcal_0$ of the square (left) and of the L-shaped domain (right)}
	\label{fig:initTriangulation}
\end{figure}
\subsubsection{Adaptive mesh refinement}\label{sec:adaptive_mesh_refinement}
The a~posteriori estimate from \Cref{thm:aposterioriGeneral} motivates the refinement-indicator
\begin{align}
\begin{split}
	\eta(T) &\coloneqq \|\sigma_h - \D W(R u_h)\|^{p'}_{L^{p'}(T)} + |T|^{p'/n}\|(1-\Pi_{\Tcal}^\ell) f\|_{L^{p'}(T)}^{p'}\\
	&\qquad\qquad + |T|^{(2-p)/p}\|R u_h - \D v_C\|^2_{L^2(T)} \quad\text{for any } T \in \Tcal
\end{split}\label{def:adaptive_refining_indicator}
\end{align}
with a postprocessing $v_C \in S^{k+1}_0(\Tcal) \coloneqq P_{k+1}(\Tcal) \cap C_0(\overline{\Omega})$ that minimizes
\begin{align*}
	\sum_{T \in \Tcal} |T|^{(2-p)/p} \|R u_h - \D w_C\|_{L^2(T)}^2 \quad\text{amongst } w_C \in S^{k + 1}_0(\Tcal).
\end{align*}
The later is a (relatively cheap) linear problem with weights to mimic the $L^p$ norm. (Since $\|R u_h - \D v_C\|^2_{L^p(T)} \approx |T|^{(2-p)/p}\|R u_h - \D v_C\|^2_{L^2(T)}$ for any $T \in \Tcal$ \cite[Lemma 4.5.3]{BrennerScott2008}, $\min_{w_{C} \in S^{k + 1}_0(\Tcal)} \|R u_h - \D w_{C}\|^{2}_{L^{p}(\Omega)} \approx  \|R u_h - \D v_C\|^{2}_{L^{p}(\Omega)}$.)
The triangulations are refined either uniformly ($\theta = 1$) or adaptively ($\theta < 1$) with the bulk parameter $\theta$: On each level $\ell = 1,2,\dots$, the adaptive algorithm marks a subset $\mathcal{M}_\ell \subset \Tcal_\ell$ (of minimal cardinality) with
\begin{align*}
\theta \sum_{T \in \Tcal_\ell} \eta(T) \leq  \sum_{K \in \mathcal{M}_\ell} \eta(K).
\end{align*}
The refinement of $\mathcal{M}_\ell$ with the newest-vertex bisection \cite{CBrenner2017,Stevenson2008} generates the new triangulation $\Tcal_{\ell+1}$.
\subsubsection{Output}
The numerical approximation of the solution to the three model problems in \Cref{sec:examples} is analysed with the focus (i) on the convergence rate of the lower energy bound (LEB) from \Cref{thm:aposterioriGeneral}.a towards the exact energy $\min E(V) - \mathrm{LEB}$ and (ii) on the a~posteriori error estimate with
\begin{align}
\text{RHS} \coloneqq E_h(u_h) - E^*(\sigma_{h}) + \osc_k(f,\Tcal_\ell) + \|R u_h-\D v_C\|^{2}_{L^{p}(\Omega)}\label{def:aposteriori}
\end{align}
from \Cref{thm:aposterioriGeneral}.b (and $v_C$ from \Cref{sec:adaptive_mesh_refinement}) and its comparison with the stress error $\|\sigma - \sigma_h\|^{2}_{L^{p'}(\Omega)}$ (if available). The uniform or adaptive mesh-refinement leads to convergence history plots of RHS, $\|\sigma - \sigma_h\|^{2}_{L^{p'}(\Omega)}$, $E(u) - \mathrm{LEB}$, and $E_h(u_h) - E^*(\sigma_h)$ against the number of degrees of freedom (ndof) displayed in \Cref{fig:pLaplaceSquareError}--\ref{fig:2WellConvUniform} below for different polynomial degrees $k$ of \Cref{fig:legend}. (Recall the scaling $\text{ndof} \propto h_{\max}^2$ in 2D for uniform mesh refinements with constant mesh-size $h_{\max}$ in a log-log plot.) In the numerical experiments without a~priori knowledge of $u$, the reference value $\min E(V)$ stems from an Aitken extrapolation of the numerical results for a sequence of uniformly refined triangulations. 
\begin{figure}[h!]
	\centering
	\includegraphics[scale=1.2]{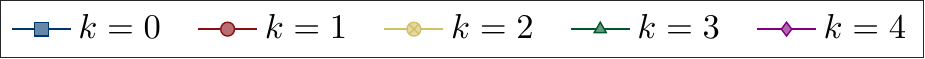}
	\caption{Polynomial degrees $k = 0,\dots,4$ in the numerical benchmarks of \Cref{sec:numericalExamples}}
	\label{fig:legend}
\end{figure}
\subsection{The p-Laplace equation}\label{sec:exPLaplace}
Let $p = 4$, $r = s = 2$, and $t = 1 + (p-2)/p = 3/2$ in the first example of \Cref{sec:examples}.
\subsubsection{Smooth solution on unit square}\label{sec:expLaplaceSquare}
Let $f \coloneqq -\div(|\D u|^{2}\D u)$ be defined by $u \in P_4(\Omega) \cap W^{1,4}_0(\Omega)$ with
\begin{align*}
	u(x_1,x_2) = x_1x_2(x_1 - 1)(x_2 - 1) \quad\text{for any } (x_1,x_2) \in \Omega = (0,1)^2.
\end{align*}
The energy functional $E$ is strictly convex, so the minimal energy $\min E(V) = E(u) = -5.10204\times10^{-04}$ is attained at the unique minimizer $u$. The interest is on the errors $\|\D u - R u_h\|_{L^4(\Omega)}$ and $\|\sigma - \sigma_h\|_{L^{4/3}(\Omega)}$.
For the smooth solution $u$ at hand, the data oscillation $\osc_k(f,\Tcal)$ in \eqref{def:aposteriori} is replaced by $\|h_\Tcal^k(1-\Pi_{\Tcal}^k)f\|_{L^{4/3}(\Omega)}$ to mimic \eqref{ineq:higherOscillation}.

\Cref{fig:pLaplaceSquareError} displays that the stress error $\|\sigma - \sigma_h\|_{L^{4/3}(\Omega)}^2$ converges optimally with convergence rates $k+1$ on uniform meshes, although \Cref{rem:pLaplace} only guarantees the convergence rates $2(k+1)/3$. The error $\|\D u - R u_h\|_{L^4(\Omega)}^2$ and RHS in \eqref{def:aposteriori} converge with the same suboptimal rates as depicted in \Cref{fig:pLaplaceSquareError}.a. For $k = 0$, the convergence rates of $\|\sigma - \sigma_h\|_{L^{4/3}(\Omega)}^2$ and $\|\D u - R u_h\|_{L^4(\Omega)}^2$ coincide, the latter is better than $1/2$ predicted in \cite{BarrettLiu1993}. Adaptive mesh refinements surprisingly recover the optimal convergence rates $k+1$ for $\|\D u - R u_h\|_{L^4(\Omega)}^2$ and RHS for any polynomial degree $k$ as depicted in \Cref{fig:pLaplaceSquareError}.b.
\Cref{fig:pLaplaceEnergyError} displays convergence rates $k+1$ for the discrete duality gap $E_h(u_h) - E^*(\sigma_h)$ and $k/2 + 1$ for $\min E(V) - \mathrm{LEB}$ on uniform and adaptive meshes.
\begin{figure}[h!]
	\begin{center}
		\begin{minipage}[t]{0.485\textwidth}
			\centering
			\includegraphics[scale=0.84]{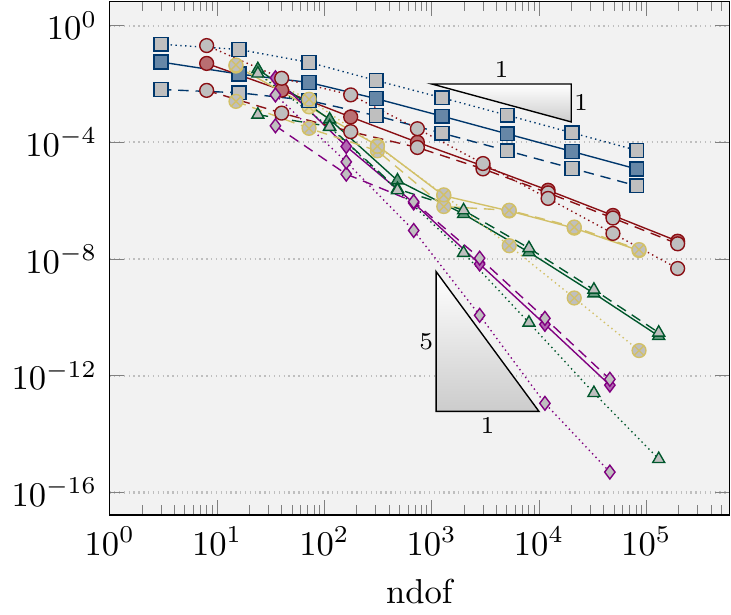}
		\end{minipage}\hfill
		\begin{minipage}[t]{0.485\textwidth}
			\centering
			\includegraphics[scale=0.84]{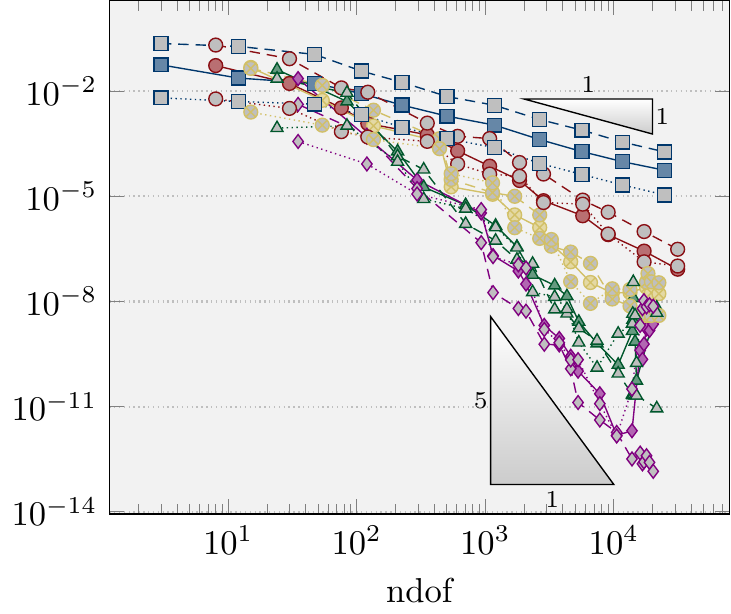}
		\end{minipage}
	\end{center}
	\caption{Convergence history plot of RHS (solid line), $\|\D u - R u_h\|^2_{L^4(\Omega)}$ (dashed line), and $\|\sigma - \sigma_h\|_{L^{4/3}(\Omega)}^2$ (dotted line) in \Cref{sec:expLaplaceSquare} for $4$-Laplace with $k$ from \Cref{fig:legend} on uniform (left) and adaptive (right) meshes}
	\label{fig:pLaplaceSquareError}
\end{figure}
\begin{figure}[h!]
	\begin{center}
		\begin{minipage}[t]{0.485\textwidth}
			\centering
			\includegraphics[scale=0.84]{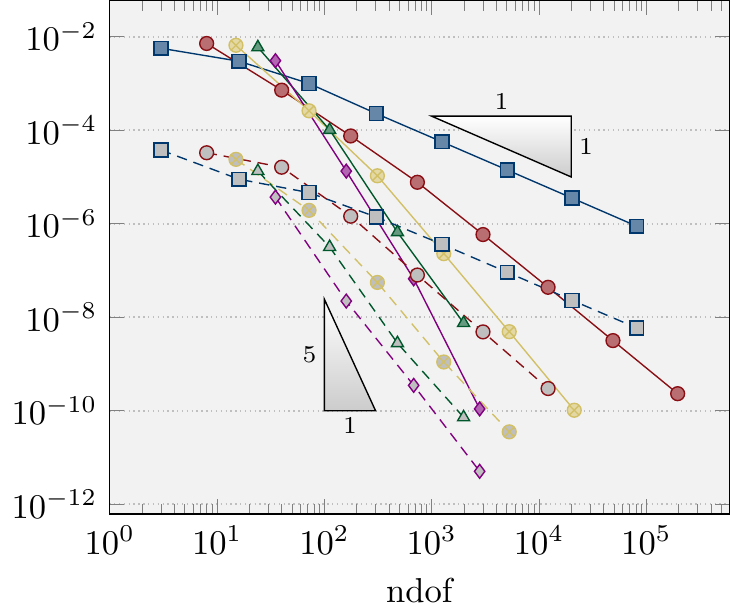}
		\end{minipage}\hfill
		\begin{minipage}[t]{0.485\textwidth}
			\centering
			\includegraphics[scale=0.84]{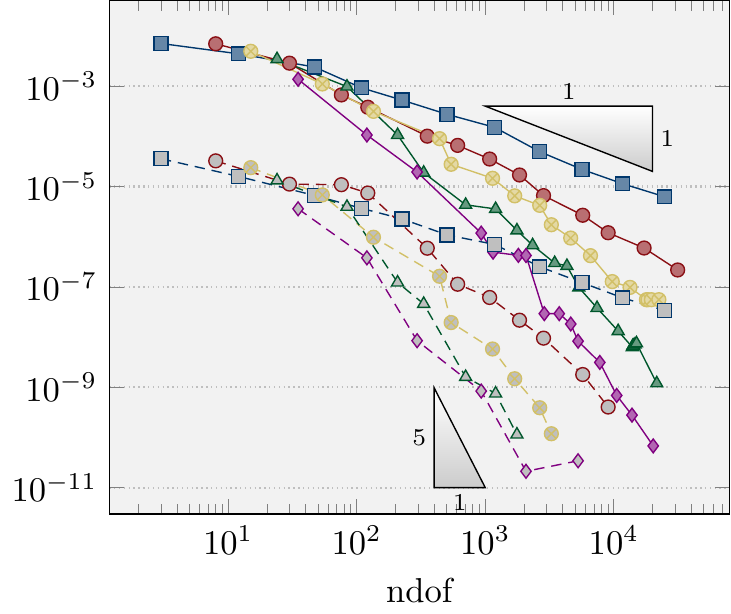}
		\end{minipage}
	\end{center}
	\caption{Convergence history plot of $E(u) - \mathrm{LEB}$ (solid line) and discrete duality gap $E_h(u_h) - E^*(\sigma_h)$ (dashed line) in \Cref{sec:expLaplaceSquare} for $4$-Laplace with $k$ from \Cref{fig:legend} on uniform (left) and adaptive (right) meshes}
	\label{fig:pLaplaceEnergyError}
\end{figure}
\subsubsection{L-shaped domain with corner singularity}\label{sec:expLaplaceLshape}
Let $\Omega = (-1,1)^2 \setminus [0,1) \times (-1,0]$ and $f \equiv 1$ with the reference value $\min E(V) = -0.34333387$. Theorem 2 in \cite{Dobrowolski1983} indicates a split $u = v + w$ of the exact solution $u$ into a singular part $v(r,\varphi) = r^\alpha t(\varphi)$ in terms of polar coordinates $(r,\varphi)$, where $w$ is a smooth function around the origin. The parameter $\alpha = (11 - \sqrt{13})/9 = 0.8216$ depends on the angle $\omega = 3\pi/2$ of the corner and $p$. The scaling $|\D u| \propto r^{\alpha - 1}$ and $|\sigma| \propto r^{(\alpha-1)(p-1)}$ indicates $\sigma \in W^{1,\beta}(\Omega;\R^n)$ for $\beta < 2/(1 - (\alpha-1)(p-1)) = 1.3028$ and we expect a convergence rate $\min\{1/2,1-1/\beta\} = 0.2324$ for the stress error $\|\sigma-\sigma_h\|_{L^{4/3}(\Omega)}^2$ on uniformly refined triangulations.
\Cref{fig:pLaplaceLshapeConv} displays a better convergence rate $0.4$ for RHS on uniform meshes. Adaptive computation refines towards the reentrant corner as depicted in \Cref{fig:pLaplaceLshapeTriangulation} and improves the convergence rate of RHS to $1$ for $k = 0$ and $2.2$ for $k = 4$.
The adaptive mesh around the singular point is much finer for larger $k$ in comparison to $k = 0$. \Cref{fig:pLaplaceLshapeEnergyConv} displays a better convergence rate of $\min E(V) - \mathrm{LEB}$ and of the discrete duality gap $E_h(u_h) - E^*(\sigma_{h})$ for $k \geq 1$.
A larger polynomial degree $k$ leads to a better convergence rate, but undisplayed computer experiments suggest that the gain is more significant for $p$ close to 2.
\begin{figure}[h!]
	\begin{minipage}[t]{0.48\textwidth}
		\centering
		\includegraphics[scale=0.75]{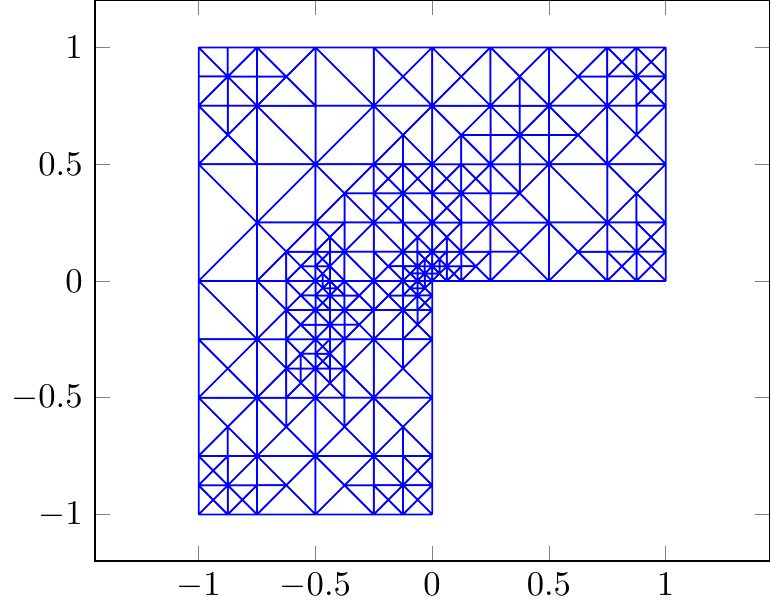}
	\end{minipage}\hfill
	\begin{minipage}[t]{0.48\textwidth}
		\centering
		\includegraphics[scale=0.75]{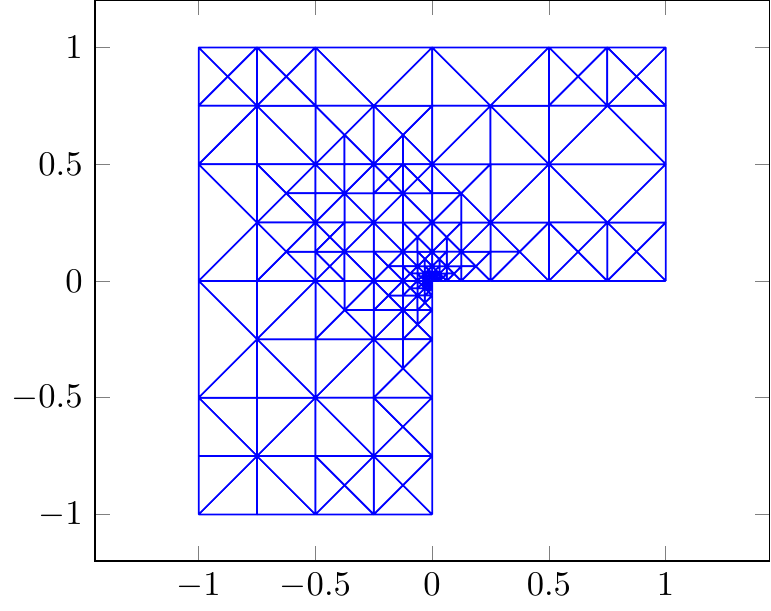}
	\end{minipage}
	\caption{Adaptive mesh of L-shape domain for $4$-Laplace in \Cref{sec:expLaplaceLshape} with (a) 431 triangles (1055 dof) for $k = 0$ (left) and (b) 481 triangles (10710 dof) for $k = 4$ (right)}
	\label{fig:pLaplaceLshapeTriangulation}
\end{figure}
\begin{figure}[h!]
	\begin{minipage}[t]{0.485\textwidth}
		\centering
		\includegraphics[scale=0.84]{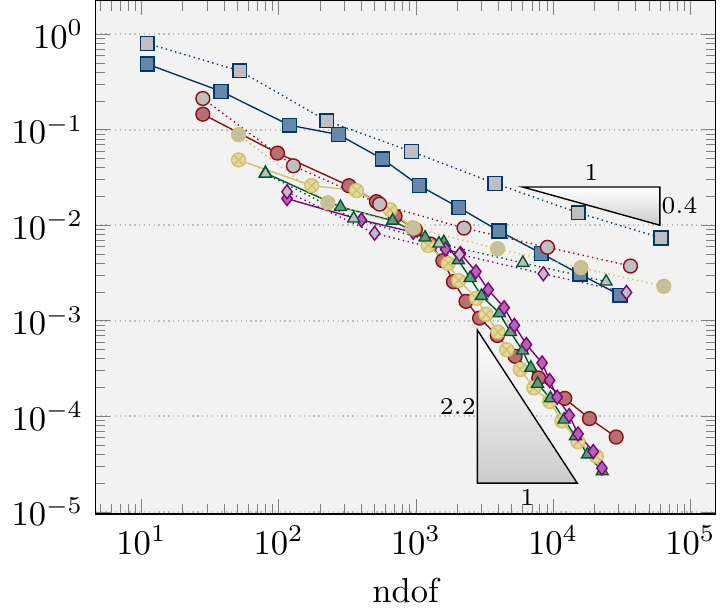}
		\caption{Convergence history plot of RHS for $4$-Laplace in \Cref{sec:expLaplaceLshape} with $k$ from \Cref{fig:legend} on adaptive (solid line) and uniform (dashed line) meshes}
		\label{fig:pLaplaceLshapeConv}
	\end{minipage}\hfill
	\begin{minipage}[t]{0.485\textwidth}
		\centering
		\includegraphics[scale=0.84]{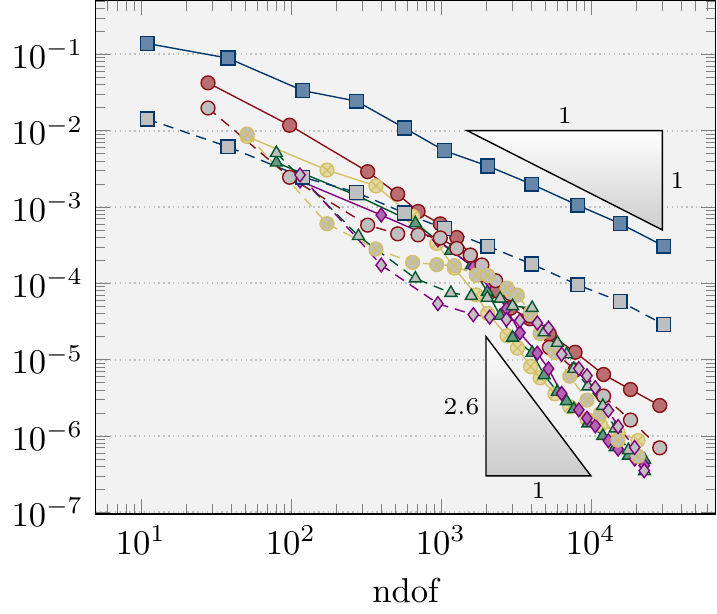}
		\caption{Convergence history plot of $\min E(V) - \mathrm{LEB}$ (solid line) and $E_h(u_h) - E^*(\sigma_h)$ (dashed line) for $4$-Laplace in \Cref{sec:expLaplaceLshape} with $k$ from \Cref{fig:legend} on adaptive meshes}
		\label{fig:pLaplaceLshapeEnergyConv}
	\end{minipage}
\end{figure}

\subsection{Optimal design problem}\label{sec:exODP}
Recall the parameters $p = r = 2$, $s = 0$, and $t = 1$ from \Cref{ex:ODP} for the optimal design problem (ODP) in topology optimization. Let $\mu_1 = 1$, $\mu_2 = 2$, $\xi_1 = \sqrt{2\lambda \mu_1/\mu_2}$ for a fixed $\lambda > 0$, $\xi_2 = \mu_2\xi_1/\mu_1$, and $f \equiv 1$. The values of $\lambda$ in the following benchmarks are from \cite[Figure 1.1]{BartelsC2008}.

\subsubsection{Material distribution and volume fraction}
The material distribution in the next two examples consists of an interior region (blue), a boundary region (yellow), and a transition layer, also called microstructure zone with a fine mixture of the two materials as depicted in \Cref{fig:ODPTriangulation}. The approximated volume fractions $\Lambda(|\Pi_{\Tcal}^0 R u_h|)$ for a discrete minimizer $u_h$ with $\Lambda(\xi) = 0$ if $0 \leq \xi \leq \xi_1$, $(\xi - \xi_1)/(\xi_2 - \xi_1)$ if $\xi_1 \leq \xi \leq \xi_2$, and $1$ if $\xi_2 \leq \xi$,
define the colour map for the fraction plot of \Cref{fig:ODPTriangulation}.
\begin{figure}[h!]
	\begin{minipage}[t]{0.48\textwidth}
		\centering
		\includegraphics[scale=0.45]{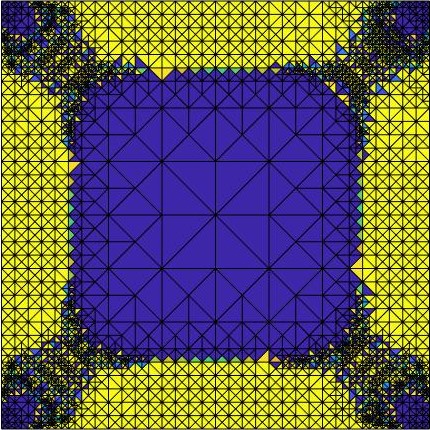}
	\end{minipage}\hfill
	\begin{minipage}[t]{0.48\textwidth}
		\centering
		\includegraphics[scale=0.45]{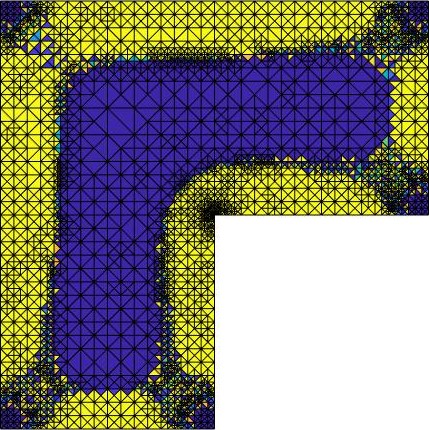}
	\end{minipage}
	\caption{Material distribution for ODP in \Cref{sec:exODP} on adaptive mesh ($k=0$) of (a) unit square (left) with 5975 triangles and of (b) L-shaped domain (right) with 5808 triangles}
	\label{fig:ODPTriangulation}
\end{figure}

\subsubsection{Unit square}\label{sec:exODPSquare}
Let $\Omega = (0,1)^2$ and $\lambda  = 0.0084$ with the reference value $\min E(V) = -0.011181337$. 
\Cref{fig:ODPConvSquare}.a shows that RHS converges with a convergence rate $3/4$ for $k = 0$ and $4/3$ for $k = 4$ 
on uniform and adaptive meshes. Higher polynomial degrees $k$ slightly improve the convergence rate of RHS.
The adaptive algorithm refines towards the microstructure zone as depicted in \Cref{fig:ODPTriangulation}.a. This leads 
to marginal improvements of the convergence rates for adaptive mesh refinements.
\Cref{fig:ODPConvSquare}.b displays larger convergence rates for the duality gap $E_h(u_h) - E^*(\sigma_{h})$ with $k \geq 1$ 
on uniform and adaptive refined triangulations.
Undisplayed numerical experiments suggest equal convergence rates of $E_h(u_h) - E^*(\sigma_{h})$ and $\min E(V) - \mathrm{LEB}$ 
and that the convergence rates of RHS, $E_h(u_h) - E^*(\sigma_{h})$, and $\min E(V) - \mathrm{LEB}$ improve with 
smaller transition layer and reaches the highest possible value $k+1$ if the measure of the transition layer vanishes.
This coincides with the numerical observations in \cite[Section 6]{CGuentherRabus2012}.
Notice that the convergence rates $k+1$ for $\min E(V) - \mathrm{LEB}$ are only possible if the oscillation $\osc_k(f,\Tcal)$ vanishes.

\subsubsection{L-shaped domain with corner singularity}\label{sec:exODPLshape}
Let $\Omega = (0,1)^2\setminus [0,1)\times(-1,0]$, and $\lambda = 0.0145$ with the reference value $\min E(V) = -0.074551285$. RHS converges suboptimally with a convergence rate $2/3$ on uniform meshes for any polynomial degree $k$ in \Cref{fig:ODPConvLshape}.a. The adaptive algorithm refines towards the reentrant corner as well as the microstructure zone in \Cref{fig:ODPTriangulation}.b. This improves the convergence rate of RHS to $3/4$ for $k = 0$ and $5/4$ for $k = 3$.
\Cref{fig:ODPConvLshape}.b shows that the convergence rates of the discrete duality gap $E_h(u_h) - E^*(\sigma_h)$ improves with larger $k$ on uniform meshes.
Similar to the previous experiments, the convergence rates of RHS, $E_h(u_h) - E^*(\sigma_h)$, and $\min E(V) - \mathrm{LEB}$ improve with higher polynomial degrees $k$, but the gain is less significant.
\begin{figure}[h!]
	\begin{minipage}[t]{0.485\textwidth}
		\centering
		\includegraphics[scale=0.84]{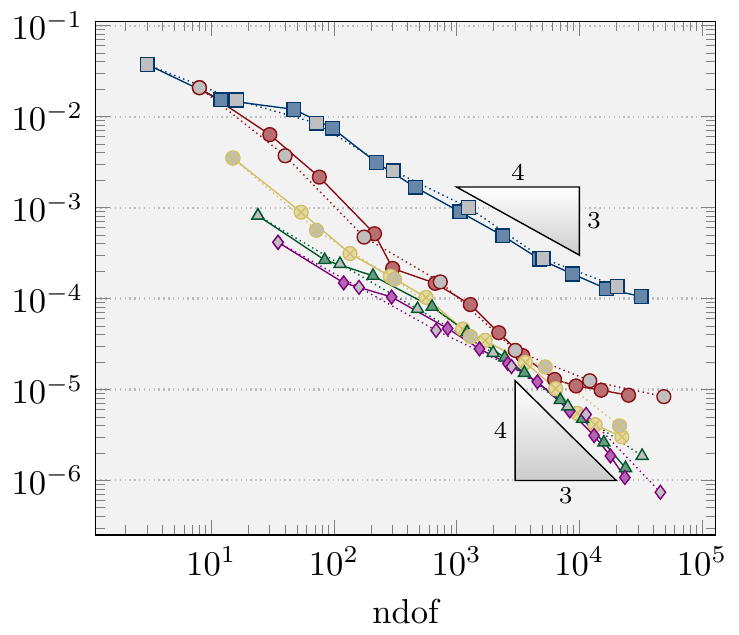}
	\end{minipage}\hfill
	\begin{minipage}[t]{0.485\textwidth}
		\centering
		\includegraphics[scale=0.84]{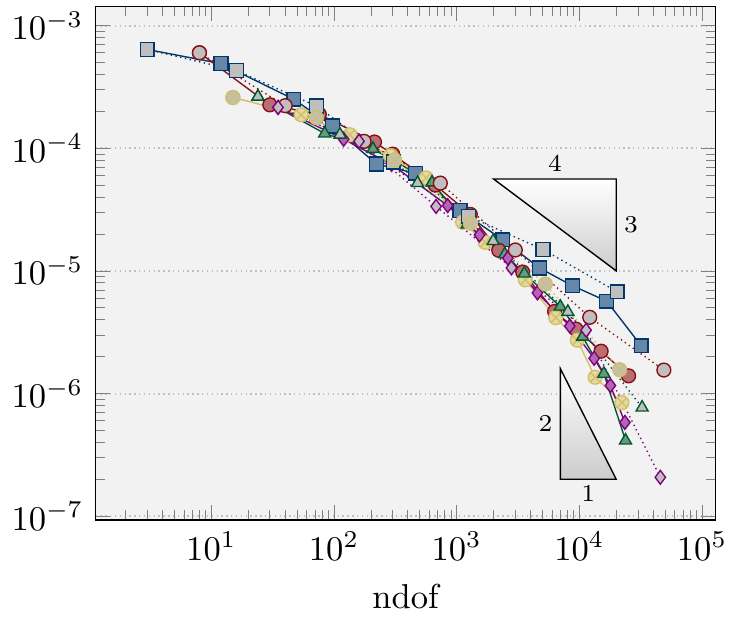}
	\end{minipage}
	\caption{Convergence history plot of RHS (left) and of $E_h(u_h) - E^*(\sigma_{h})$ (right) for ODP in \Cref{sec:exODPSquare} on adaptive (solid line) and uniform (dashed line) meshes}
	\label{fig:ODPConvSquare}
\end{figure}

\begin{figure}[h!]
	\begin{minipage}[t]{0.485\textwidth}
		\centering
		\includegraphics[scale=0.84]{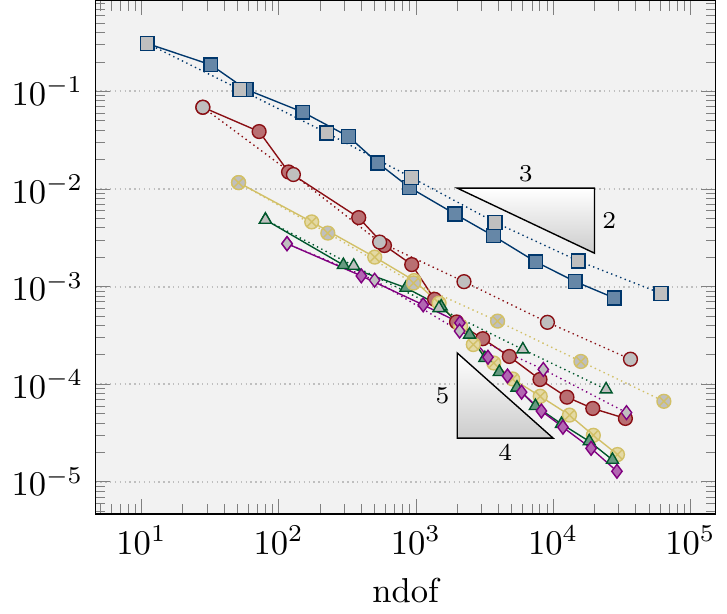}
	\end{minipage}\hfill
	\begin{minipage}[t]{0.485\textwidth}
		\centering
		\includegraphics[scale=0.84]{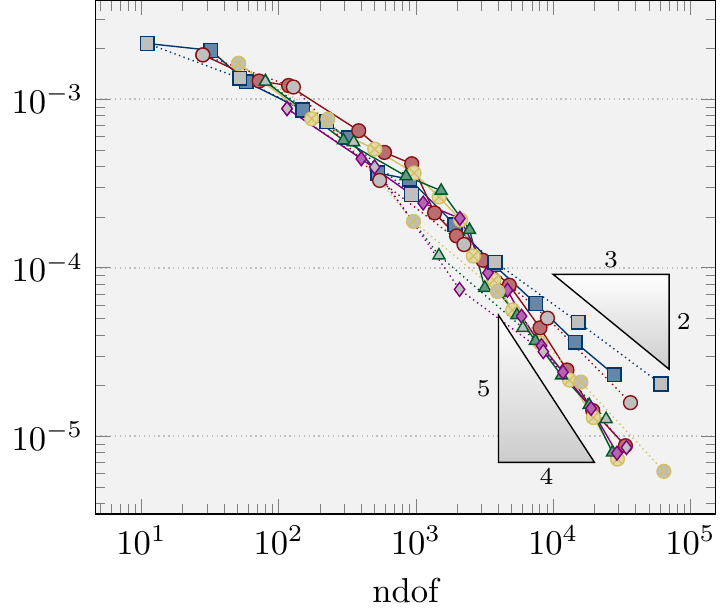}
	\end{minipage}
	\caption{Convergence history plot of RHS (left) and of $E_h(u_h) - E^*(\sigma_{h})$ (right) for ODP in \Cref{sec:exODPLshape} on adaptive (solid line) and uniform (dashed line) meshes}
	\label{fig:ODPConvLshape}
\end{figure}

\subsection{Two-well computational benchmark}\label{sec:ex2well}
The computational benchmark in \cite{CPlechac1997,CJochimsen2003} involves an additional quadratic term $1/2\|g - u_\Tcal\|^2_{L^2(\Omega)}$ in \eqref{intro:energy}, that leads to uniqueness of the continuous minimizer $u$ and of the volume component $u_\Tcal$ of the discrete minimizer $u_h = (u_\Tcal,u_\Fcal)$, and inhomogeneous Dirichlet data $u_\mathrm{D}$. \Cref{thm:apriori} and \Cref{thm:aposteriori} can be extended to the situation at hand for the discrete dual energy
\begin{align*}
	E_d^*(\tau_h) &\coloneqq -\int_{\Omega} \big(W^*(\tau_h) + g \cdot (\div \tau_h + \Pi_{\Tcal}^k f)\big) \d{x}  + \int_{\partial \Omega} u_\mathrm{D} \cdot (\tau_h \nu) \d{s}\\
	&\qquad\qquad - \frac{1}{4\alpha}\|\div \tau_h + \Pi_{\Tcal}^k f\|^2_{L^2(\Omega)} \quad\text{for any } \tau_h \in \Sigma(\Tcal) \cap W^{p'}(\div,\Omega;\M).
\end{align*}
(The quadrature formula of \Cref{sec:numerical_realization}   computates the integral of  $W^*(\tau_h)$ 
and $W^*$ is evaluated pointwise by the MATLAB routine \texttt{fminunc} with high accuracy for $\texttt{FunctionTolerance} = \texttt{OptimalityTolerance} = \texttt{StepTolerance} = 10^{-14}$).
The precise data for $u_\mathrm{D}$, $f$, $g$, and $u$ can be found in \cite[page 179]{CJochimsen2003}.
The exact solution $u$ on $\Omega = (0,1) \times (0,3/2)$ is piecewise smooth and $\D u$ jumps across the interface $S = \operatorname{conv}\{(1,0),(0,3/2)\}$. 
The initial triangulation $\Tcal_0$ consists of two triangles with the interiors in $\Omega \setminus S$.

The extension of \Cref{thm:apriori} and \Cref{thm:aposteriori} leads to error estimates for $\|\sigma - \sigma_{h}\|^2_{L^{4/3}(\Omega)} + \|u - u_\Tcal\|^2_{L^2(\Omega)}$ that predict optimal convergence rates. Those are confirmed in \Cref{fig:2WellConvUniform} for $\|\sigma - \sigma_{h}\|^2_{L^{4/3}(\Omega)}$, $\|\D u - R u_h\|_{L^4(\Omega)}^2$, $\|u - u_\Tcal\|^2_{L^2(\Omega)}$, and $|E(u) - E_h(u_h)|$, and polynomial degrees $k = 0,\dots,3$ with rates $k+1$ and a very accurate discrete solution for $k = 4$. The 
modified a~posteriori estimate 
 $\widehat{\mathrm{RHS}} \coloneqq \mathrm{RHS} + \|h_\Tcal(1 - \Pi_{\Tcal}^k) g\|_{L^2(\Omega)}$
 and the guaranteed lower energy bound $\min E(V) - \mathrm{LEB}$ converge with the rates $k/2 + 1$ in \Cref{fig:2WellConvUniform}.b. 
 Undisplayed numerical experiments with the adaptive algorithm indicate no improvements of the convergence with optimal rates on the structured meshes of \Cref{fig:2WellConvUniform}.

\begin{figure}[h!]
	\begin{minipage}[t]{0.485\textwidth}
		\centering
		\includegraphics[scale=0.84]{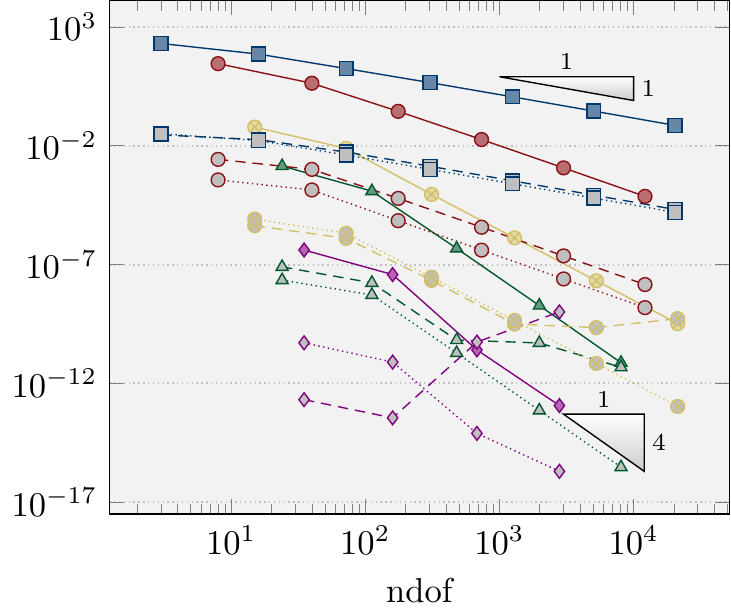}
	\end{minipage}\hfill
	\begin{minipage}[t]{0.485\textwidth}
		\centering
		\includegraphics[scale=0.84]{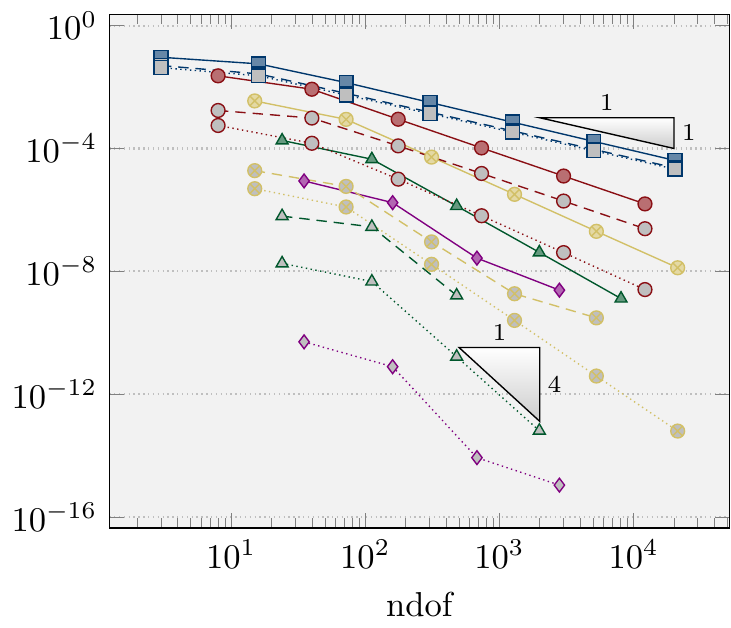}
	\end{minipage}
	\caption{Convergence history plot of $\|\sigma - \sigma_{h}\|^2_{L^{4/3}(\Omega)}$ (solid line), $\|\D u - R u_h\|^2_{L^{4}(\Omega)}$ (dashed line), $\|u - u_\Tcal\|^2_{L^2(\Omega)}$ (dotted line) (left) and of $\widehat{\mathrm{RHS}}$ (solid line), $\min E(V) - \mathrm{LEB}$ (dashed line), $|E(u) - E_h(u_h)|$ (dotted line) (right) for 2-well problem in \Cref{sec:ex2well} with $k$ from \Cref{fig:legend}}
	\label{fig:2WellConvUniform}
\end{figure}

\subsection{Conclusions}
The computer experiments provide empirical evidence for the improved convergence rates of the unstabilized HHO methods for examples of 
degenerate convex minimization. The numerical results confirm the theoretical findings,  
in particular,  the suggested guaranteed lower energy bounds are confirmed bounds and converge superlinearly to the exact energy $\min E(V)$ in all examples. 
The a posteriori estimate in Subsection 4.2 motivates an adaptive mesh-refining algorithm for the HHO schemes that converges in the examples. 
A higher polynomial degree $k$ leads to  improved convergence rates of the stress error. 
Optimal convergence rates are observed for the (piecewise) smooth solutions in 
\Cref{sec:expLaplaceSquare}--\ref{sec:ex2well}.

\bibliographystyle{siamplain}
\bibliography{references}
\end{document}